\documentclass[11pt,a4paper]{article}
\usepackage{amsmath,amsfonts,amssymb,amsthm,amsbsy}
\usepackage{appendix}
\usepackage{mathrsfs}
\usepackage{enumitem}
\usepackage[colorlinks,linkcolor=blue,anchorcolor=blue,citecolor=blue]{hyperref}
\usepackage{color,xcolor}
\usepackage{tabularx}
\usepackage{authblk}
\usepackage{bbm}
\usepackage{algorithm,algorithmic}
\usepackage{breakcites}
\usepackage{graphicx}

\makeatletter
\newenvironment{breakablealgorithm}
  {
    \begin{center}
      \refstepcounter{algorithm}
      \hrule height.8pt depth0pt \kern2pt
      \parskip 0pt
      \renewcommand{\caption}[2][\relax]{
        {\raggedright\textbf{\fname@algorithm~\thealgorithm} ##2\par}%
        \ifx\relax##1\relax 
          \addcontentsline{loa}{algorithm}{\protect\numberline{\thealgorithm}##2}%
        \else 
          \addcontentsline{loa}{algorithm}{\protect\numberline{\thealgorithm}##1}%
        \fi
        \kern2pt\hrule\kern2pt
     }
  }
  {
     \kern2pt\hrule\relax
   \end{center}
  }
\makeatother

\usepackage[left=1in,right=1in,top=1.2in,bottom=1.2in]{geometry}
\numberwithin{equation}{section}
\newtheorem{definition}{Definition}[section]
\newtheorem{theorem}{Theorem}[section]
\newtheorem{lemma}{Lemma}[section]
\newtheorem{remark}{Remark}[section]
\newtheorem{corollary}{Corollary}[section]

\newtheorem{proposition}{Proposition}[section]

\newcommand{\Y}{\mathcal{Y}}
\newcommand{\mloc}{M^\mathrm{ loc } }

\title{Minimax Rates for Robust Gaussian Mean Estimation \\ Near the Breakdown Point under Constraints}
\author[*]{Tuorui Peng}
\author[*]{Matey Neykov}
\affil[*]{Department of Statistics and Data Science, Northwestern University}
\date{\today}

\begin{document}

\maketitle

\begin{abstract}
  We study the mean estimation of Gaussian location model with star-shaped constraints under adversarial contamination, where at most $ k<N/2 $ among $ N $ observations are corrupted by an oracle adversary. We characterize the minimax squared $ \ell_2 $ risk throughout the full range of contamination level $ \varepsilon = k/N \in [0,1/2) $. The minimax risk is
  \begin{align*}
    \mathfrak{M}\asymp \big( \delta _\varepsilon ^{*2} \vee \sigma ^2\varepsilon ^2\log \frac{e}{1-2\varepsilon } \big) \wedge d^2 ,
  \end{align*}
  where $ \delta _\varepsilon ^* := \sup  \big\{ r \geq 0: (1-2\varepsilon )\frac{Nr ^2}{\sigma ^2} \leq \log M^\mathrm{ loc }_K(r ,c) \big\} $, $ \log\mloc_K (r ,c) $ is the local entropy of constraint set $ K $, $ d $ is the diameter of $ K $ (set $ d=\infty $ if $ K $ is unbounded), $ \sigma ^2 I_n $ is the covariance of the Gaussian noise, and $ c $ is a sufficiently large constant. Our result sheds light on the interaction between contamination error and information theoretic complexity in the regime of proximity to the breakdown point $ \varepsilon \uparrow 1/2 $.
 
\end{abstract}

\tableofcontents

\section{Introduction}

Robust statistics under contamination is a central topic in modern statistics, and is drawing increasing attention from both the statistics and computer science communities \cite{diakonikolas_robustly_2018,prasadan_information_2026}. A fundamental model in this field is the location model with adversarial corruption, where $ k $ in $ N $ samples are corrupted by an adversarial algorithm \cite{diakonikolas_robustly_2018}. Call the contamination fraction $ \varepsilon := k/N $. The fraction $ \varepsilon = 1/2 $ is the breakdown point for unique recovery of the mean parameter: beyond it, two different mean parameters can generate the same corrupted data and are thus indistinguishable. When the contamination level is bounded away from the breakdown point, as assumed in much of the existing literature, the minimax rate is governed by the usual local entropy scale. Work such as \cite{prasadan_information_2026} treats the contamination level $ \varepsilon \leq 1/2-\kappa $ for some fixed $ \kappa >0 $. Crucially, this restriction prevents the analysis of the minimax rate in the regime of proximity to the breakdown point $ \varepsilon \uparrow 1/2 $. The purpose of this paper is to fill this gap and characterize the minimax rate in the full range of $ \varepsilon \in [0,1/2) $.

\subsection{Problem Setup}

We start by introducing the adversarial corruption model as follows, as in \cite{diakonikolas_robustly_2018}[Definition 2.5].

\begin{definition}[Gaussian Location Model under Adversarial Corruption]\label{def:adversarial_corruption_model}
Let clean data 
\begin{align*}
  \tilde{X}_1,\ldots,\tilde{X}_N \mathop{ \sim }\limits^{i.i.d.} \mathcal{N}(\mu,\sigma^2 I_n),\quad \mu \in K\subseteq \mathbbm{R}^n.  
\end{align*}
Then an oracle adversarial algorithm $ \mathcal{C} $ is allowed to examine and then corrupt at most $ k $ clean samples, where $ 0\leq k < \frac{N}{2} $; we write $ \varepsilon := k/N $ as the fraction of corrupted data. The adversary has access to the model parameters, the clean data, and the mapping of the estimator, using auxiliary randomness if needed. The resulting corrupted data is denoted as $ X_{1:N} := (X_1,\ldots,X_N) =\mathcal{C}(\tilde{X}_{1:N}) $.
\end{definition}

Our goal is to characterize the minimax risk of estimating $ \mu $ under the adversarial corruption model, in terms of squared $ \ell_2 $ loss up to absolute constant factors. Define the minimax risk:
\begin{align}\label{equation:minimax_risk}
  \mathfrak{M} = \mathfrak{M}(K,N,k,\sigma) := \inf_{\hat{\mu}}\sup_{\mu\in K}\sup_{\mathcal{C}}\mathbbm{E}_{\mu ,\mathcal{C} }\left[ \left\Vert \hat{\mu}(X_{1:N})-\mu  \right\Vert ^2 \right],
\end{align}
where the risk is defined for given $ K,N,k,\sigma $, and the estimator $ \hat{\mu} $ is allowed to depend on these parameters. The corruption adversary $ \mathcal{C} $ can possibly have auxiliary randomness, and the expectation is taken over both the randomness of the clean data and the adversary.

\subsection{Notation}

Throughout the paper, we use $ n $ for the ambient dimension, $ N $ the sample size, $ k $ the number of corrupted samples, $ \varepsilon := k/N $ the contamination fraction, $ \sigma ^2 I_n $ the covariance of the Gaussian noise, and $ K\subseteq \mathbbm{R}^n $ the parameter constraint set. Define the diameter of $ K $ as $ d:=\operatorname{diam}(K):=\sup_{x,y\in K}\left\Vert x-y \right\Vert $.

For any two real numbers $ a,b $, we define $ a\vee b := \max\{a,b\} $ and $ a\wedge b := \min\{a,b\} $. For any integer $ N $, we define $ [N] := \{1,2,\ldots,N\} $. For any vector $ x\in \mathbbm{R}^n $, we denote its $ \ell_2 $ norm as $ \left\Vert x \right\Vert := \sqrt{\sum_{i=1}^n x_i^2} $. For any matrix $ A $, we denote its transpose as $ A^\top $. Write $ B(x,r):= \{y\in \mathbbm{R}^n: \left\Vert y-x \right\Vert \leq r\} $ as the closed ball centered at $ x $ with radius $ r $. Write $ B_q(R) := \{x\in \mathbbm{R}^n: \left\Vert x \right\Vert _q\leq R\} $ as the $ \ell_q $ ball with radius $ R $ centered at the origin.
For nonnegative quantities $ a,b $ depending on the problem parameters $ n,N,k,\sigma ,K $, we write $ a \lesssim b $ if $ a \leq Cb $ for some universal constant $ C>0 $ independent of $ n,N,k,\sigma ,K $. Similarly, we write $ a \gtrsim b $ if $ a \geq Cb $, and $ a\asymp b $ if both $ a\lesssim b $ and $ a\gtrsim b $. 
For any set $ A $, we denote its indicator function as $ \mathbbm{1}_A(x) = 1 $ if $ x\in A $, and $ 0 $ otherwise. We denote the normal distribution with mean $ \mu  $ and covariance matrix $ \Sigma $ as $ \mathcal{N}(\mu ,\Sigma ) $, and the binomial distribution with parameters $ m,p $ as $ \mathrm{Binomial}(m,p) $. For any two probability measures $ P,Q $, we denote their total variation distance as $ d_\mathrm{ TV }(P,Q) := \sup_{A}|P(A)-Q(A)| $. 

Also, we formalize the definition of two key concepts in this work: star-shaped set and local metric entropy.
\begin{definition}[Star-Shaped Set]
  A set $ K\subseteq \mathbbm{R}^n $ is called star-shaped if there exists a point $ x^*\in K $ such that for any $ x\in K $, the entire line segment connecting $ x^* $ and $ x $ is contained in $ K $. Formally, for any $ t\in [0,1] $, we have $ tx^* + (1-t)x \in K $. Any such point $ x^* $ is called a center of the star-shaped set $ K $.
\end{definition}

\begin{remark}
   Star-shaped sets are a generalization of convex sets. Every convex set is star-shaped,  with every point in the set being a center. However, not every star-shaped set is convex, making the star-shaped set a substantially richer class of sets. An important instance of a star-shaped set is the sparse vector set, which is non-convex but star-shaped with the origin being a center. 
\end{remark}

\begin{definition}[Local Metric Entropy]
  Given a set $ K\subseteq \mathbbm{R}^n $ and $ r >0 $, the $ r  $-packing number of $ K $ is defined as the supremum of the cardinality $ M = M(r ,K) $ of a subset $ \{\nu_1,\ldots,\nu_M\}\subseteq K $ such that $ \left\Vert \nu_i-\nu_j \right\Vert \geq r $ for all $ i\neq j $. Then given any fixed constant $ c>1 $, define the maximal local packing number of $ K $ as
  \begin{align*}
    \mloc_K(r,c):= \mathop{ \sup }\limits_{\nu\in K} M(r/c,K\cap B(\nu,r)),
  \end{align*}
  where we may drop the subscript $ K $ when it is clear from the context. The local metric entropy of $ K $ is then defined as $ \log \mloc_K(r,c) $. We make the convention that $ \log \mloc_K(0,c) := 0$.
\end{definition}

\subsection{Related Work}

The theoretical study of robust statistics has a long history, dating back to Huber's seminal work \cite{huber_robust_1964}. Later, \cite{hampel_general_1971} formalized the notion of breakdown point, and \cite{bickel_festschrift_1982-1} identified $ 1/2 $ to be the natural breakdown point for location estimation. These classical works determined where the estimation problem becomes statistically impossible, while leaving the finite-sample behavior largely unexplored. In recent years, more corruption regimes have been studied, and there have been more works developing both computational and theoretical foundations of robust estimation \cite{chen_general_2016,diakonikolas_robustly_2018,lugosi_robust_2021,kotekalOptimalEstimationNull2025,prasadan_information_2026}. We outline several closely related works in the following.

Robust mean estimation near the breakdown point has been studied in prior work. \cite{chen_general_2016} developed the minimax theoretic framework for Huber's contamination model, and showed how robust testing yields optimal rates across various estimation models. \cite{kotekalOptimalEstimationNull2025} studied the Gaussian mean-shift model and established the minimax rate for the whole range of contamination levels. Their result found a phase transition governed by $ N-2k $, and the established rate can be substantially faster than those under unrestricted contamination. \cite{chen_outlier-robust_2025} focused on adversarial corruption model similar to ours through an overlap-based analysis and considered robust sparse mean estimation. However, their results do not provide a general characterization of the minimax rate under general parameter constraints. Our work further connects the corruption component with the information-theoretic component induced by local geometry, capturing the interaction between the constraint geometry and adversarial corruption.

For mean estimation with constraints under adversarial corruption, we utilize a tournament-on-the-tree framework first developed in \cite{neykovMinimaxRateGaussian2023} and later extended to adversarial corruption in \cite{prasadan_information_2026}. Crucially, \cite{prasadan_information_2026} extended the framework to adversarial corruption with star-shaped constraints, and established the minimax rate for $ \varepsilon $ strictly bounded away from $ 1/2 $. As a result, the rate dependence when approaching the breakdown point was not characterized. 

Going beyond the breakdown point of $ 1/2 $, unique recovery is no longer possible, and list-decodable learning becomes a natural alternative \cite{charikar_learning_2017,diakonikolas_list-decodable_2018}.

\subsection{Contributions and Organization}

Our main contribution is a sharp characterization of the
minimax squared $\ell_2$ risk for Gaussian mean estimation
under adversarial corruption with star-shaped constraints,
uniformly over the full range of unique recovery of the mean parameter, i.e., $\varepsilon \in [0,1/2)$. We extend the fixed-gap results of \cite{prasadan_information_2026} and show that the proximity to the breakdown point affects the minimax rate in two ways: (i) the information-theoretic component is governed by a fixed point equation involving the local metric entropy and the effective sample size $N-2k = N(1-2\varepsilon)$, and (ii) the contamination component exhibits a logarithmic divergence of order $\sigma^2 \varepsilon^2 \log \frac{e}{1-2\varepsilon}$. Our characterization of the minimax rate applies to both bounded and unbounded star-shaped sets, and recovers the existing results when $\varepsilon$ is strictly bounded away from $1/2$.

Our proof establishes matching upper and lower bounds through several key arguments. 
The upper bound is based on the tournament-on-the-tree framework in \cite{prasadan_information_2026}, with careful calibration of constants in the tail probabilities and stopping conditions when $ \varepsilon  $ approaches $ 1/2 $. For the unbounded case, we further assess the localization step and its impact on the risk.
On the lower bound side, we first introduce a posterior symmetrization trick to establish the effective sample size reduction from $ N $ to $ N-2k $ with a proper adversary. This produces the information-theoretic lower bound via a local Fano's argument. Then we construct a truncated maximal-coupling adversary and apply Le Cam's two-point method to establish the contamination lower bound. 
Together, these arguments establish matching bounds up to universal constants independent of contamination level $ \varepsilon $, allowing us to characterize the behavior of the adversarial corruption minimax rate in the full regime of $ \varepsilon \in [0,1/2) $.

The remainder of the paper is organized as follows. \autoref{section:main_results} presents the main results, including the minimax rate and several examples, and outlines the lower- and upper-bound arguments. \autoref{section:discussion} discusses possible extensions and future directions. The appendices provide technical details: \autoref{appendix:tournament_on_tree_framework} provides a brief review of the tournament-on-the-tree framework and several useful properties involved in the analysis; \autoref{appendix:proof_of_examples} derives the rates for the examples in \autoref{subsection:examples}; \autoref{appendix:lower_bound_proof} proves the lower bound and \autoref{appendix:constant_calibration_near_breakdown_point}-\ref{appendix:minimax_upper_bound_unbounded_case} prove the upper bound.

\section{Main Results}\label{section:main_results}

\begin{theorem}[Minimax Rate]\label{main_theorem:minimax_rate}
    Let $ K\subseteq \mathbbm{R}^n $ be a star-shaped set with diameter $ d $ (in the case of unbounded $ K $, we set $ d=\infty $). Set $ 0\leq k < N/2 $ and $ \varepsilon := k/N $. Let samples be drawn following \autoref{def:adversarial_corruption_model}. Then 
    \begin{align}\label{equation:main_result_minimax_rate}
      \mathfrak{M}\asymp \big( \delta _\varepsilon ^{*2} \vee \sigma ^2\varepsilon ^2\log \frac{e}{1-2\varepsilon } \big) \wedge d^2.
    \end{align}
    Here $ \delta _\varepsilon ^* $ is defined as
    \begin{align}\label{equation:critical_local_radius_definition}
      \delta _\varepsilon ^*:= \mathop{ \sup }\big\{ r \geq 0: (1-2\varepsilon )\frac{Nr ^2}{\sigma ^2} \leq \log M^\mathrm{ loc }(r ,c) \big\},  
    \end{align}
    in which $ c $ is a sufficiently large constant.    
\end{theorem}

\autoref{main_theorem:minimax_rate} reveals two mechanisms behind the minimax rate, together with their dependence on the corruption level $ \varepsilon  $: The first term $ \delta _\varepsilon ^{*2} $ is the information theoretic component, which is a function of the local metric entropy $ \log\mloc(r ,c) $ and the effective sample size $ N-2k=N(1-2\varepsilon ) $. The second term $ \sigma ^2\varepsilon ^2\log \frac{e}{1-2\varepsilon } $ is the contamination error component, exhibiting a logarithmic rate divergence as $ \varepsilon  $ approaches $ 1/2 $. 
This agrees with its counterpart for Huber's contamination model \cite{kotekalOptimalEstimationNull2025}[Appendix E], but differs from the sharper rate for Gaussian mean-shift model.

Our work is a direct continuation of \cite{prasadan_information_2026}. Our result recovers their Gaussian rate when $ \varepsilon \in [0,1/2-\kappa ] $ is strictly bounded away from $ 1/2 $, and further extends the minimax rate to the whole range of $ \varepsilon \in [0,1/2) $. We formalize the recovery of their fixed-gap result in the following proposition. Proof given in \autoref{appendix:subsection:useful_properties}.

\begin{proposition}\label{proposition:agreement_with_existing_result}
  Fix an absolute constant $ 0<\kappa <1/2 $ (for example, $ \kappa = 1/8 $) and suppose that $ 0\leq\varepsilon \leq 1/2-\kappa $. Choose the same sufficiently large constant $ c $ in both \autoref{main_theorem:minimax_rate} and \cite{prasadan_information_2026}[Theorem 3.10]. Let the critical local packing radius $ \eta ^* $ from their work be 
  \begin{align*}
     \eta^* := \sup\big\{ r\geq 0: \frac{Nr^2}{\sigma ^2} \leq \log M^\mathrm{ loc }(r ,c) \big\}.
  \end{align*}
  Then we have $ \eta^* \leq \delta _\varepsilon ^* \leq \frac{\eta^*}{\sqrt{2\kappa }} $. As a result,
  \begin{align*}
    \mathfrak{M}\asymp (\eta^{*2}\vee \sigma ^2\varepsilon ^2) \wedge d^2 \asymp (\delta _\varepsilon ^{*2}\vee \sigma ^2\varepsilon ^2 \log\frac{e}{1-2\varepsilon }) \wedge d^2,
  \end{align*}
  and similarly for the unbounded case $ d=\infty $, reducing to the existing result in \cite{prasadan_information_2026}. 
\end{proposition}

\subsection{Examples}\label{subsection:examples}

We illustrate the minimax rate in \autoref{main_theorem:minimax_rate} with several examples of constraint sets $ K $, including the unconstrained case, the $ \ell_2 $ and $ \ell_1 $ balls, and the sparse vector set. In each bounded example, we use the radius $ R:= d/2 $ to quantify the size of the constraint set. 

\paragraph{Unconstrained Case} For $ K = \mathbbm{R}^n $, the volume argument of \autoref{lemma:useful_properties}[(ii)] gives $ \log\mloc_K(r ,c) \asymp n $ for any $ r >0 $. Hence
\begin{align*}
  \mathfrak{M}(\mathbbm{R}^n) \asymp \frac{\sigma ^2n}{N(1-2\varepsilon )} \vee \sigma ^2\varepsilon ^2\log\frac{e}{1-2\varepsilon }.
\end{align*}

\paragraph{Euclidean Ball} For $ K = B_2(R) := \{x\in \mathbbm{R}^n: \left\Vert x \right\Vert \leq R\} $, the same local metric entropy holds for $ 0<r \lesssim R $ and the only difference is the radius cap.
\begin{align*}
   \mathfrak{M}(B_2(R)) \asymp\big( \frac{\sigma ^2n}{N(1-2\varepsilon )} \vee \sigma ^2\varepsilon ^2\log\frac{e}{1-2\varepsilon } \big) \wedge R^2.
\end{align*}

\paragraph{$ \ell_1 $ Ball} Consider $ K = B_1(R) := \{x\in \mathbbm{R}^n: \left\Vert x \right\Vert _1 \leq R\} $, where we assume $ n\geq 2 $. The local metric entropy scales as follows, as in \cite{prasadan_facts_2025}[Lemma 3.9]: when $ 0<r\leq R $,
\begin{align*}
  \log\mloc_{B_1(R)}(r ,c)\asymp \begin{cases}
    n,&r \leq R/\sqrt{n}\\
    \frac{R^2}{r ^2}\log\frac{enr ^2}{R^2} ,& R/\sqrt{n}< r \leq R.
  \end{cases} 
\end{align*}
The resulting minimax rate is
\begin{align*}
  \mathfrak{M}(B_1(R)) \asymp \begin{cases}
    \big( R^2\sqrt{\frac{\log (en/\sqrt{\lambda })}{\lambda }} \vee \sigma ^2\varepsilon ^2\log\frac{e}{1-2\varepsilon } \big) \wedge R^2,& \lambda \leq n^2\\
    \big( R^2\frac{n}{\lambda } \vee \sigma ^2\varepsilon ^2\log\frac{e}{1-2\varepsilon } \big) \wedge R^2,& \lambda > n^2,
  \end{cases}
\end{align*}
where we denote $ \lambda := \frac{N(1-2\varepsilon )R^2}{\sigma ^2} $. Details of the derivation are given in \autoref{appendix:proof_of_examples}.

\paragraph{Sparse Vector Set} Let $ 1\leq s \leq n/8 $ and let $ K $ be the unbounded sparse vector set $ K := \{x\in \mathbbm{R}^n: \left\Vert x \right\Vert _0 \leq s\} $. The local metric entropy scales as $ \log\mloc_{K}(r ,c)\asymp s\log\frac{en}{s} $ for any $ r >0 $. It follows that
\begin{align*}
  \mathfrak{M}(\{x\in \mathbbm{R}^n: \left\Vert x \right\Vert _0 \leq s\}) \asymp \frac{\sigma ^2s\log\frac{en}{s}}{N(1-2\varepsilon )} \vee \sigma ^2\varepsilon ^2\log\frac{e}{1-2\varepsilon }.
\end{align*}

In the case of an extra radius constraint $ R>0 $, i.e. $ K := \{x\in \mathbbm{R}^n: \left\Vert x \right\Vert _0 \leq s, \left\Vert x \right\Vert _2 \leq R\} $, the same packing number holds with an additional radius cap, and the minimax rate is
\begin{align*}
  \mathfrak{M}(\{x\in \mathbbm{R}^n: \left\Vert x \right\Vert _0 \leq s, \left\Vert x \right\Vert _2 \leq R\}) \asymp \big( \frac{\sigma ^2s\log\frac{en}{s}}{N(1-2\varepsilon )} \vee \sigma ^2\varepsilon ^2\log\frac{e}{1-2\varepsilon } \big) \wedge R^2.
\end{align*}

\subsection{Lower Bound}

We split the lower bound into the regimes of $ \varepsilon \leq 3/8 $ and $ \varepsilon > 3/8 $. The former is covered by \cite{prasadan_information_2026}; the latter is the main contribution of this work. The lower bound consists of two components: the information theoretic component and the contamination component. We present the key steps to establish each of them when $ \varepsilon  $ is close to $ 1/2 $, details of the proofs are deferred to \autoref{appendix:lower_bound_proof}. 

\paragraph{Information Theoretic Component} We first look at the information-theoretic component of the minimax rate. The key observation is that, with a proper adversary, an adversarial corruption problem with $ k $ in $ N $ corrupted samples is at least as hard as a clean data problem with $ N-2k $ samples. This induces a reduction of the effective sample size, resulting in the information-theoretic component of the minimax rate. The following lemma formalizes this argument.
  
\begin{lemma}[Effective Sample Size Reduction]\label{lemma:effective_sample_size_reduction}
  Let $ \Theta := \{\mu _1,\ldots,\mu _M\} \subset K $ be finite, and let $ V\sim \pi$ be a prior over $ [M] $. Write $ \mu _V $ for the corresponding mean. The clean data is generated as
  \begin{align*}
    \tilde{X}_1,\ldots, \tilde{X}_N \mathop{ \sim }\limits^{i.i.d.} \mathcal{N}(\mu _V,\sigma ^2I_n),\quad \mu _V\in \Theta .  
  \end{align*}
  Given an adversarial corruption algorithm $ \mathcal{C} $ that corrupts the clean data $ \tilde{X}_{1:N} $ into $ X_{1:N} $ with at most $ k=\varepsilon N < N/2 $ corrupted samples, let
  \begin{align*}
    R_\mathrm{ corrupt }(\hat{\mu }, \mathcal{C}, \pi ):= \mathbbm{E}_{\mu_V \sim \pi}\left[ \mathbbm{E}_{\mu_V, \mathcal{C} }\left[ \left\Vert \hat{\mu}( X_{1:N})- \mu_V  \right\Vert ^2 \right] \right].
  \end{align*}
  Let the Bayesian risk for clean data of size $ m $ be
  \begin{align*}
    R_\mathrm{ clean }(m,\pi) := \inf_{\hat{\mu}}\mathbbm{E}_{\mu_V \sim \pi}\left[ \mathbbm{E}_{ \mu_V }\left[ \left\Vert \hat{\mu}( \tilde{X}_{1:m})- \mu_V  \right\Vert ^2 \right] \right].
  \end{align*}

  Then for any prior $ \pi $, there exists an adversarial corruption algorithm $ \mathcal{C} _\pi $ such that for every estimator $ \hat{\mu} $ it holds that
  \begin{align*}
     R_\mathrm{ corrupt }(\hat{\mu }, \mathcal{C}_\pi , \pi ) \geq \frac{1}{2} R_\mathrm{ clean }(N-2k,\pi).
  \end{align*}

\end{lemma}

Then applying \autoref{lemma:effective_sample_size_reduction} to a local packing set and using Fano's method, we can establish the information theoretic component lower bound. 

\begin{remark}
  \autoref{lemma:effective_sample_size_reduction} provides a useful connection between the estimation under adversarial corruption and the clean-sample estimation, illustrating the effective sample size reduction from $ N $ to $ N-2k $. The posterior symmetrization step itself is not specific to the Gaussian location family or the squared $ \ell_2 $ loss. Gaussianity enters later, when the clean-sample Bayesian risk is quantified through the Gaussian KL calculation in the local Fano's argument. This suggests that the posterior symmetrization may provide a useful tool for establishing information-theoretic lower bounds in other adversarial corruption problems: once the loss-specific relation between the posterior pairwise loss and the clean Bayesian risk is established, the same exchangeability construction can yield an analogous reduction.
\end{remark}

\paragraph{Contamination Component} We now turn to the contamination component of the minimax rate. This can be done by carefully assessing the total variation distance between two Gaussian distributions to show that these two distributions are indistinguishable under a suitable adversary. 

\begin{lemma}\label{lemma:contamination_lower_bound_adversary_construction}
  Let $ a $ be some large enough absolute constant and assume $ N $ is large enough so that $ \varepsilon _- := \varepsilon -\frac{a}{\sqrt{N}} >1/4 $. Let $ \mu _0,\mu _1\in K $ such that
  \begin{align*}
    d_\mathrm{ TV }( \mathcal{N}(\mu _0,\sigma ^2I_n), \mathcal{N}(\mu _1,\sigma ^2I_n)) \leq 2\varepsilon _-.  
  \end{align*}
  Then there exist two adversaries $ \mathcal{C}_0, \mathcal{C}_1 $, each corresponding to $ \mu _0 $ and $ \mu _1 $, such that if letting
  \begin{align*}
    \tilde{P}^{(N)}_b := \mathcal{L}\big( \mathcal{C}_b(\tilde{X}_{1:N}) \big),\quad b=0,1, 
  \end{align*}
  be the distribution of the corrupted data, we have $ d_\mathrm{ TV }( \tilde{P}^{(N)}_0, \tilde{P}^{(N)}_1) < \frac{1}{2} $, meaning that it is strictly bounded away from $ 1 $. As a result, we have
  \begin{align*}
    \mathfrak{M}\geq \mathop{ \inf }\limits_{\hat{\mu }}\mathop{ \sup }\limits_{b\in \{0,1\}} \mathbbm{E}_{ \tilde{P}_b ^{(N)}}\left[ \left\Vert \hat{\mu} - \mu_b  \right\Vert ^2 \right] \gtrsim \left\Vert \mu _0-\mu _1 \right\Vert ^2.
  \end{align*}  
\end{lemma}

\begin{remark}
  Here the constant $ a $ is chosen universally. As a result, the lemma holds for sufficiently large $ N $, say $ N\geq N_a $ depending only on $ a $; the remaining finite sample size case can be handled by analyzing the range of $ \varepsilon  $ and see that it reduces to the small $ \varepsilon  $ case, as shown in the proof of \autoref{lemma:contamination_lower_bound}.
\end{remark}

The characterization of the total variation distance between two Gaussian distributions was given in previous work, say \cite{kotekalOptimalEstimationNull2025}[Appendix E]. This allows us to establish the contamination component of the minimax lower bound, as formalized in the following lemma.

\begin{lemma}[Contamination Lower Bound]\label{lemma:contamination_lower_bound}
  Let $ \mathfrak{M} $ be the minimax risk defined in \autoref{main_theorem:minimax_rate} and $ \varepsilon =\frac{k}{N}\in(3/8,1/2) $. Then we have
  \begin{align*}
    \mathfrak{M}\gtrsim \sigma ^2\varepsilon ^2\log\frac{1}{1-2\varepsilon } \wedge d^2,
  \end{align*}
  and set $ d=\infty $ in the case of unbounded $ K $.

\end{lemma}
  
The above two components characterize the minimax lower bound when $ \varepsilon  $ is close to $ 1/2 $. Together with \autoref{proposition:agreement_with_existing_result}, which covers $ \varepsilon \in [0,3/8] $, we can establish the minimax lower bound in the full range of $ \varepsilon \in [0,1/2) $.

\begin{theorem}[Minimax Lower Bound]\label{theorem:minimax_lower_bound}
  Let $ \varepsilon = \frac{k}{N} \in [0,1/2) $, let $ \mathfrak{M} $ be the minimax risk defined in \autoref{main_theorem:minimax_rate}, and let $ d $ be the diameter of $ K $. Then we have
  \begin{align*}
    \mathfrak{M}\gtrsim \max\{\delta _\varepsilon^{*2} , \sigma ^2\varepsilon ^2\log(e/(1-2\varepsilon )) \} \wedge d^2,
  \end{align*}
  and set $ d=\infty $ in the case of unbounded $ K $. We remind that here $ \delta _\varepsilon ^* $ is the critical radius defined as in \autoref{equation:critical_local_radius_definition}.
\end{theorem}

\subsection{Upper Bound}

For the upper bound, we apply the tournament-on-the-tree framework as in \cite{prasadan_information_2026}[Section 3]. The tree structure and estimator construction remain unchanged up to a rescaling, while the main difference is that we need to calibrate the constants in the tail probability bound when $ \varepsilon  $ approaches $ 1/2 $ to characterize the minimax rate near the breakdown point. Specifically, the error exponent of pairwise robust testing procedure becomes $ \varepsilon  $-dependent, together with a different stopping condition for the tournament procedure. For unbounded $ K $, the $ \varepsilon  $-dependence in the localization argument also needs to be assessed. 

In the following we present the main results for the upper bound. We refer the readers to \autoref{appendix:subsection:pruned_directed_tree} for a review of the tournament-on-the-tree framework and \autoref{appendix:constant_calibration_near_breakdown_point}-\ref{appendix:minimax_upper_bound_unbounded_case} for the technical details of the upper bound proof.

\begin{theorem}[Minimax Upper Bound for Bounded $ K $]\label{theorem:minimax_upper_bound_bounded}
  Let $ K $ be a bounded star-shaped set with diameter $ d $. Suppose the samples are drawn following \autoref{def:adversarial_corruption_model} where $ \varepsilon = \frac{k}{N} \in [0,1/2) $. Then we have the following minimax upper bound:
  \begin{align}\label{equation:minimax_upper_bound_bounded}
    \mathfrak{M}\lesssim \max \{ \delta _\varepsilon ^{*2} , \sigma ^2\varepsilon ^2\log(e/(1-2\varepsilon )) \} \wedge d^2,
  \end{align}
  where we define $ \delta _\varepsilon ^* $ as in \autoref{equation:critical_local_radius_definition}.
\end{theorem}

For unbounded $ K $, there is an extra step to localize the problem to a bounded region. We argue that the localization probability and the empty localization set probability have the same contribution and do not affect the minimax rate. As a result, we have the following similar argument as in \cite{prasadan_information_2026}[Section 5]. 

\begin{theorem}[Minimax Upper Bound for Unbounded $ K $]\label{theorem:minimax_upper_bound_unbounded}
  Let $ K $ be an unbounded star-shaped set. Suppose the samples are drawn following \autoref{def:adversarial_corruption_model} where $ \varepsilon = \frac{k}{N} \in [0,1/2) $. Then we have the following minimax upper bound:
  \begin{align*}
    \mathfrak{M}\lesssim \max\{\delta _\varepsilon ^{*2} , \sigma ^2\varepsilon ^2\log(e/(1-2\varepsilon )) \},
  \end{align*}
  where $ \delta _\varepsilon ^* $ is defined as in \autoref{theorem:minimax_upper_bound_bounded}.
\end{theorem}

\section{Discussion}\label{section:discussion}

We established the minimax rate for constraint Gaussian mean estimation under adversarial corruption, with a focus on the regime when the corruption level $ \varepsilon  $ approaches the breakdown point of $ 1/2 $. There are two distinct effects of proximity to the breakdown point on the minimax rate: A decrease in the effective sample size to $ N(1-2\varepsilon ) $ changes the critical local metric entropy radius $ \delta _\varepsilon ^* $; and a logarithmic blow-up $ \varepsilon ^2\log\frac{e}{1-2\varepsilon } $ in the contamination error component. When $ \varepsilon  $ is strictly bounded away from $ 1/2 $, our result recovers the existing minimax rate in \cite{prasadan_information_2026}.

Our result utilizes the tournament-on-the-tree framework in \cite{prasadan_information_2026} to establish the minimax upper bound. They also presented the application of the framework to sub-Gaussian model, and the subsequent work \cite{peng_robust_2026} extended to even heavier tailed distributions of finite moments. It is of interest to characterize the full near-breakdown behavior under these broader noise settings. Another interesting direction is Gaussian noise with known covariance $ \Sigma  $. The proof of our work does not translate directly to this setting: though the data can be directly whitened, the loss and the information-theoretic component induced by the constraint set are no longer invariant under the whitening transformation. As a result, the current formulation does not extend automatically to the known covariance case, and a more careful analysis is needed. Also, our estimation allows dependence on known $ \sigma ^2 $. Establishing the minimax rate with unknown $ \sigma ^2 $ information is another possible future direction. We leave these extensions to future work.

\section*{Declaration of AI Use}

During the preparation of this manuscript, the authors used the generative AI tools from OpenAI to assist with literature review, grammar and language polishing, and refinement of mathematical arguments. Generative AI was also used in the verification of the constants in certain proofs. All AI-assisted suggestions were evaluated by the authors, who take full responsibility for the mathematical results, proofs, and other content of this paper.

\bibliography{NearBreakdown}
\bibliographystyle{apalike}

\appendix

\section{Tournament-on-the-Tree Framework}\label{appendix:tournament_on_tree_framework}

The tournament-on-the-tree framework provides a systematic approach to establish the upper bounds of robust estimation problems. The core idea was first introduced in \cite{neykovMinimaxRateGaussian2023} and later extended in \cite{prasadan_information_2026} to handle more general settings. We briefly summarize the key steps of this framework below and refer the readers to \cite{peng_robust_2026}[Section 3] for a more systematic introduction. Throughout the upper bound arguments, we use an $ \varepsilon  $ subscript in a parameter to indicate its dependency on the corruption level. 

We retain the large enough constant $ C>2 $ as the separation constant in the pairwise robust testing, and set $ c:= 4(C+1) $ as the constant in the local metric entropy definition and the tree construction, which follows the re-notation from \cite{prasadan_information_2026}[Remark 3.9]. 
For bounded $ K $ with diameter $ d $, let the geometric scale be
\begin{align}\label{equation:geometric_scale_r_J}
  r_J:= \frac{d}{2^{J-1}(C+1)},\quad J\geq 1. 
\end{align}

\subsection{Useful Properties}\label{appendix:subsection:useful_properties}

We first collect several useful properties of the local packing number and star-shaped sets. 

\begin{lemma}[Useful Properties]\label{lemma:useful_properties}
  Let $ K\subseteq \mathbbm{R}^n $ be a star-shaped set and let $ c>1 $ be some constant. Then the following properties hold:
  \begin{itemize}[topsep=2pt,itemsep=0pt]
    \item[(i)] The mapping $ r\mapsto \mloc_K(r,c) $ is non-increasing in $ r>0 $;
    \item[(ii)] For any $ r>0 $, we have $ \mloc_K(r,c) \leq (1+2c)^n $;
    \item[(iii)] If $ K $ is bounded with diameter $ d $, then $ K $ must contain a line segment of length $ \geq d/3 $. If $ K $ is unbounded, then it must contain line segments of any length.
  \end{itemize}
\end{lemma}

\begin{proof}
  The proofs of (i) and (iii) were given in \cite{prasadan_information_2026}[Lemma 1.4, Lemma 1.3]. (ii) is a direct consequence of the volume argument, see \cite{peng_robust_2026}[Remark 4.2] for details. 
\end{proof}

We also give the following proof of \autoref{proposition:agreement_with_existing_result} to show that our result is consistent with \cite{prasadan_information_2026} when $ \varepsilon  $ is strictly bounded away from $ 1/2 $.

\begin{proof}[Proof of \autoref{proposition:agreement_with_existing_result}]
  We prove for the bounded $ K $ case, and the unbounded $ K $ case follows similarly. 
    
  By \autoref{lemma:useful_properties}[(i)], for any $ r< \delta _\varepsilon ^* $ we have
  \begin{align*}
    \log\mloc(\sqrt{1-2\varepsilon }r,c) \geq \log\mloc(r,c) \geq (1-2\varepsilon )Nr^2/\sigma ^2 = N(\sqrt{1-2\varepsilon }r)^2/\sigma ^2,
  \end{align*}
  meaning that $ \sqrt{1-2\varepsilon }r \leq \eta ^* $. Taking supremum over $ r<\delta _\varepsilon ^* $ gives the desired sandwich inequality. 
  Finally, noticing that
  \begin{align*}
    1\leq \log\frac{e}{1-2\varepsilon } \leq 1+\log\frac{1}{2\kappa } ,
  \end{align*}
  so the two minimax rates are equivalent up to a constant factor depending on $ \kappa  $.
\end{proof}

\subsection{Pruned Directed Tree Construction and the Estimator}\label{appendix:subsection:pruned_directed_tree}

Fix a sufficiently large universal separation constant $ C>2 $ and set $ c:= 4(C+1) $. Here $ c $ is the one to be used in local metric entropy through the proof. We remind that we adapt their \cite{prasadan_information_2026}[Remark 3.9] to rescale $ c $ thus it is different from the original $ c_{\mathrm{orig}}=2(C+1) $ by a factor of $ 2 $. For bounded $ K $ with diameter $ d $, let the geometric scale be
\begin{align*}
  r_J = &  \frac{d}{2^{J-1}(C+1)},\quad J\geq 1. 
\end{align*}

Starting from an arbitrary root in $ K $, we construct the pruned directed tree as in \cite{prasadan_information_2026}[Section 3.1]. Let $ \mathcal{L}(J) $ denote level $ J $ of the tree, and let $ \mathcal{O}(u) $ be the offspring of node $ u $.
In their construction, the offspring at each level are obtained from a maximal local packing set of the appropriate ball around the parent node, followed by a pruning step to ensure that the offspring are well-separated. The following is the formal algorithm for the construction of the pruned directed tree, and then we summarize the properties of the pruned directed tree adopted from \cite{prasadan_information_2026}[Lemma 3.1 -- 3.3] in \autoref{lemma:pruned_directed_tree_properties}. The constants in the construction are slightly different from the original \cite{prasadan_information_2026}[Section 3.1] due to the rescaling of $ c $.

\begin{breakablealgorithm}
    \renewcommand{\algorithmicrequire}{\textbf{Input:}}
    \renewcommand{\algorithmicensure}{\textbf{Output:}}
    \caption{Directed Tree Construction}
    \label{algorithm:directed_tree_construction}
   \begin{algorithmic}[1]
        \REQUIRE Set $ K $, Root node $ s\in K $, diameter $ d>0 $, constant $ c = 4(C+1) $ chosen large enough;

        \ENSURE A directed tree structure $ \mathcal{G} $ where each layer $ J $ is denoted $ \mathcal{L}(J) $;

        \STATE $ \mathcal{L}(1)\leftarrow \{s\} $;

        \STATE Form a maximal $ (2d/c) $-packing set of $ B(s,d)\cap K $, then add directed edges from $ s $ to each node in the packing set, i.e., the packing set is the offspring of $ s $, denoted $ \mathcal{O}(s) $;

        \STATE $ \mathcal{L}(2)\leftarrow \mathcal{O}(s)   $;

        \FOR{$ J = 3,4,\ldots $}

            \FORALL{$ u \in \mathcal{L}(J-1)  $}

                \STATE Form a maximal $ (d/2^{J-2}c) $-packing set of $ B(u, d/2^{J-2}) \cap K $, then add directed edges from $ u $ to these nodes, denoted as $ \mathcal{O}(u) $;

            \ENDFOR

            \STATE $ \mathcal{U}_J \leftarrow $ lexicographically ordered set of all offspring nodes $ \mathcal{O}(u) $ for all $ u\in \mathcal{L}(J-1) $, as constructed in the previous step;

            \WHILE{$ \mathcal{U}_J \neq \varnothing $}

                \STATE Pick first element, say $ u_l^J $, of $ \mathcal{U}_J $;

                \STATE $ \mathcal{T}_J(u_l^J) \leftarrow \{u_k^J\in \mathcal{U}_J:\, \left\Vert u_k^J-u_l^J \right\Vert \leq \frac{d}{2^{J-2}c},\, k\neq l \} $;

                \STATE For each $ u_k^J \in \mathcal{T}_J(u_l^J) $, remove directed edge $ \mathcal{P}(u_k^J) \to u_k^J $ and the node $ u_k^J $ from $ \mathcal{G} $, add edge $ \mathcal{P}(u_k^J) \to u_l^J $;

                \STATE Remove $ \{u_l^J\} \cup \mathcal{T}_J(u_l^J) $ from $ \mathcal{U}_J $;

            \ENDWHILE

            \STATE $ \mathcal{L}(J)\leftarrow \bigcup_{u\in \mathcal{L}(J-1)} \mathcal{O}(u) $;            

        \ENDFOR

        \RETURN $ \mathcal{G} $.
    \end{algorithmic} 
\end{breakablealgorithm}

\begin{lemma}[Pruned Directed Tree Properties]\label{lemma:pruned_directed_tree_properties}
  For the above described pruned directed tree, we have the following properties:
  \begin{itemize}[topsep=2pt,itemsep=0pt]
    \item[(i)] For any $ J\geq 3 $, $ \mathcal{L}(J) $ is an $ r_J $-covering and an $ (r_J/2) $-packing of $ K $;
    \item[(ii)] For any $ J\geq 2 $ and any $ u\in \mathcal{L}(J-1) $, the offspring $ \mathcal{O}(u) $ form an $ r_J $-covering of $ K\cap B(u,(c/2)r_{J}) $, and $ \left\vert \mathcal{O}(u) \right\vert \leq \mloc_K((c/2)r_J,c) $;
    \item[(iii)] For any $ J\geq 2 $ and any $ \mu \in K $, we have
    \begin{align*}
      \left\vert \mathcal{L}(J-1) \cap B(\mu ,(c/2)r_J) \right\vert \leq \mloc_K((c/2)r_J,c);  
    \end{align*}
    \item[(iv)] Every infinite path $ \Y = [\Y_1,\Y_2,\ldots] $ where $ \Y_{J+1} \in \mathcal{O}(\Y_J) $ satisfies: for $ 1\leq J'\leq J $,
    \begin{align*}
      \left\Vert \Y_J - \Y_{J'} \right\Vert \leq \frac{4C+5}{2}r_{J'} 
    \end{align*}
  \end{itemize}
\end{lemma}

\paragraph{Tournament Based on Robust Comparison}

Another key ingredient of the tournament-on-the-tree framework is a tournament procedure based on the robust pairwise comparison over the local packing set at each layer. We remind the readers that the tree is constructed such that for any node, its offspring form a local packing set of the appropriate ball around the parent node. The objective of the tournament is to select a winner which can best represent the true mean parameter $ \mu $ among the candidates. This boils down to a robust pairwise comparison between each pair of candidates. We now briefly summarize the procedure that builds up the tournament. 

The fundamental building block of the tournament is the robust pairwise comparison between two candidates. Given an ordered pair of candidates $ (\nu _1,\nu _2) $ where $ \nu _1,\nu _2\in K $ and $ \left\Vert \nu _1-\nu _2 \right\Vert \geq Cr $, we seek to test whether the true mean $ \mu  $ of $ \tilde{X} $ belongs to $ B(\nu _1,r) $ or $ B(\nu _2,r) $. The testing function is defined as
\begin{align*}
  \psi = \psi_{(\nu _1,\nu _2)} = \begin{cases}
    0 ,& \text{if deciding } \mu \in B(\nu _1,r),\\
    1 ,& \text{if deciding } \mu \in B(\nu _2,r).
  \end{cases} 
\end{align*}
Assume $ \nu _1, \nu _2 $ in $ \mathbbm{R}^n $ are in lexicographical order. If $\psi_{(\nu_1,\nu_2)}=0$, we write $\nu_1\succ\nu_2$; otherwise, we write $\nu_2\succ\nu_1$. For this paper, the robust testing we use is the one specified in \cite{prasadan_information_2026}[Definition 3.4]:
\begin{definition}
  Given any ordered pair of candidates $ (\nu _1,\nu _2) $ and under the model in \autoref{def:adversarial_corruption_model}, the testing function is defined by
  \begin{align*}
    \psi_{(\nu _1,\nu _2)}(X_{1:N}) = \mathbbm{1}\big( \mathrm{card}\{i\in [N]: \left\Vert X_i-\nu _1 \right\Vert \geq \left\Vert X_i-\nu _2 \right\Vert \} \geq \frac{N}{2} \big).
  \end{align*}
  For simplicity, we write $ \psi $ when the notation is clear from the context.
\end{definition}
The robust pairwise comparison leads to a robust tournament among the offspring of each node in the pruned directed tree. Based on the directed tree constructed, a robust tournament series can be performed. We record one step here for clarity. On each level $ J\geq 2 $ of the tree, the winner of the previous level is denoted as $ \Y_{J-1} $, and its offspring $ \mathcal{O}(\Y_{J-1}) $ are the candidates for the next node. Call this candidate set $ S = \mathcal{O}(\Y_{J-1}) $ and for $ \nu \in S $, let
\begin{align*}
  T(r,\nu ,S) = \begin{cases}
    \mathop{ \max }\limits_{\nu '\in S'} \left\Vert \nu -\nu ' \right\Vert ,& \text{if } S'\neq \varnothing,\\
    0,& \text{if } S' = \varnothing, 
  \end{cases}
\end{align*}
where we define $ S' = \{\nu '\in S: \nu ' \succ \nu \text{ and } \left\Vert \nu -\nu ' \right\Vert \geq Cr \} $. Let $ \hat{\nu } $ be any minimizer of $ T(r,\nu ,S) $ over $ S $ and thus the winner of the tournament at level $ J $. If each pairwise comparison makes an error with probability at most $ p(r) $, then the error of each tournament winner can be bounded as follows, which is adapted from \cite{prasadan_information_2026}[Lemma 3.7].

\begin{lemma}[Tournament Error Bound]\label{lemma:tournament_error_bound}
  Let $ S \subseteq K $ be a finite $ r $-covering of $ K'\subseteq K $, and $ \mu \in K' $. Assume that for every pair $ \nu _1,\nu _2\in S $ with $ \left\Vert \nu _1-\nu _2 \right\Vert \geq Cr $, the pairwise comparison makes either error in \autoref{lemma:robust_pairwise_testing_near_breakdown_point} with probability at most $ p(r) $. Then the winner $ \hat{\nu } $ of the tournament satisfies
  \begin{align*}
    \mathbbm{P}_{ \mu  }\left( \left\Vert \hat{\nu } - \mu  \right\Vert \geq (C+1)r \right) \leq \left\vert S \right\vert\cdot p(r) .  
  \end{align*}  
\end{lemma}

\paragraph{Bounded $ K $}

In the case of bounded $ K $, we can directly apply \autoref{algorithm:directed_tree_construction} to obtain the tree structure. Then by running the tournament along the tree, we have a sequence of winners and form a chain $ \Y = [\Y_1,\Y_2,\ldots] $ along the tree, given the root node. The following is a formal statement of the tournament series algorithm.

\begin{breakablealgorithm}
    \renewcommand{\algorithmicrequire}{\textbf{Input:}}
    \renewcommand{\algorithmicensure}{\textbf{Output:}}
    \caption{Robust Tournament Algorithm for an Upper Bound.}
    \label{algorithm:robust_algorithm}

    \begin{algorithmic}[1]
        \REQUIRE Directed tree $ \mathcal{G} $;

        \ENSURE A winner chain $ \Y=[\Y_1,\Y_2,\ldots] $ in $ \mathcal{G} $;

        \STATE Initialize $ J=1 $, $ \Y_1=s $ and $ \Y = [\Y_1] $;

        \FOR{$ J=1,2,\ldots $}

            \STATE $ \Y_{J+1}\leftarrow \mathop{ \arg\min }\limits_{\nu \in \mathcal{O}(\Y_J)} T\big( \frac{d}{2^J(C+1)}, \,\nu ,\, \mathcal{O}(\Y_J) \big) $, breaking ties lexicographically;

            \STATE $ \Y.\text{append}(\Y_{J+1}) $;

        \ENDFOR

        \RETURN $ \Y=[\Y_1,\Y_2,\ldots] $.
    \end{algorithmic}
\end{breakablealgorithm}

As we will see later in the proof of the upper bound, traversing enough steps along the winner chain $ \Y $ ensures a close enough estimate to the true mean parameter $ \mu $. 

\paragraph{Unbounded $ K $}

In the case of unbounded $ K $, there is no natural diameter to define the geometric scale. The problem is addressed by a localization argument, where a large enough radius $ R_\varepsilon  $ is chosen within which the data concentrates around the true mean parameter $ \mu $ with high probability. The radius may depend on $ n,\sigma ,\varepsilon  $ and the universal constant $ C $. As a result, it induces a random set $ S(R) $ based on the observed data which contains the true mean parameter $ \mu $ with high probability. Another key construction is an appropriate covering and packing set of $ K $ built independent of the data, which facilitates a forest of trees, each rooted at a node in the set. The two components act together to restrict the search space of the tournament to a finite set, and we have one tree to work with following the procedure in the bounded $ K $ case. 

The localization set $ S(R) $ is a data dependent subset of $ K $ defined as follows.
\begin{definition}[Localization Set]\label{definition:localization_set_SR}
  For each fixed $ \nu  $ and $ i\in [N] $, let event $ E_{\nu ,i}(R):= \{\left\Vert X_i-\nu  \right\Vert > R\} $ and define
  \begin{align*}
    S(R) := \big\{ \nu \in K: \sum_{i=1}^N \mathbbm{1}(E_{\nu ,i}(R)) < N/2 \}. 
  \end{align*}
  Thus, $ S(R) $ consists of points for which strictly more than half of the data points are within distance $ R $ from it. 

  Unlike \cite{prasadan_information_2026}, where they used the threshold $ \leq N/2-1 $, we use $ < N/2 $ instead to include the case when $ N $ is odd. This change does not affect the properties of the localization set, nor the proof of the minimax upper bound.
\end{definition}

The set $ S(R) $ has the following properties from \cite{prasadan_information_2026}[Lemmas 5.2 and 5.4]:
  \begin{itemize}[topsep=2pt,itemsep=0pt]
    \item $ S(R) $ is nested in $ R$, i.e., $ S(R_1)\subseteq S(R_2) $ for $ R_1<R_2 $;
    \item $ \mathrm{ diam }(S(R)) \leq 2R $.
  \end{itemize} 

The countable packing of $ K $ is constructed under the assumption that $ K $ is closed, otherwise we take the closure of $ K $. It is argued in \cite{prasadan_information_2026}[Remark 5.13] that the minimax rate remains unchanged with the closure replacement. The constant adjustment in the argument is also universal and do not involve $ \varepsilon  $-dependence. Then set $ m:= \frac{R_\varepsilon}{c/2-1}$ (we remind the readers that we have rescaled $ c $ by a factor of $ 2 $ from the original one), and construct an $ m $-packing and $ 2m $-covering set $ S_m $ of $ K $, which can possibly be countably infinite. The existence of such a packing and covering set is guaranteed by the fact that $ K $ is closed. On each point $ s $ in the packing set, we construct a pruned directed tree as in \autoref{algorithm:directed_tree_construction}, where we set the input geometric scale as $ d=d_m = 2m+2R_\varepsilon  $ and restrict the region of interest to $ B(s,d_m)\cap K $. Denote each resulting tree as $ \mathcal{G}^s $, each layer by $ \mathcal{L}^s(J) $, etc., where the superscript $ s $ indicates the tree rooted at $ s $. We remark that the forest of trees $ \{ \mathcal{G}^s: s\in S_m \} $ is data-independent. Following the construction, we have the counterpart of \autoref{lemma:pruned_directed_tree_properties} for each tree $ \mathcal{G}^s $ in the forest. 

The output for the unbounded $ K $ case is divided into the following two subcases.
\begin{itemize}[topsep=2pt,itemsep=0pt]
  \item If $ S(R_\varepsilon) = \varnothing $, then we take the estimator as the lexicographically first point in $ S(\hat{R}) $ where $ \hat{R} = \inf\{R\geq 0: S(R)\neq \varnothing \} $;
  \item If $ S(R_\varepsilon) \neq \varnothing $, then we pick the lexicographically first point $ s\in S_m $ such that $ s $ is closest to $ S(R_\varepsilon) $, and run the tournament along the tree $ \mathcal{G}^s $ to obtain a winner chain $ \Y^s = [\Y_1^s,\Y_2^s,\ldots] $. A node after enough steps along the winner chain is taken as the final estimator.
\end{itemize}

\section{Proof of Examples}\label{appendix:proof_of_examples}

The computation for unconstrained $ K $, Euclidean ball $ K $, and sparse vector classes are straightforward by directly calculating the critical local metric entropy radius $ \delta _\varepsilon ^* $ and applying \autoref{main_theorem:minimax_rate}. Below we provide the details for the $ \ell_1 $ ball case, which is more involved.

\paragraph{Proof of $ \ell_1 $ ball case} It suffices to compute the critical local metric entropy radius $ \delta _\varepsilon ^* $ for the $ \ell_1 $ ball. The local metric entropy is given by
\begin{align*}
  \log\mloc _{B_1(R)}(r,c) \asymp \begin{cases}
    \text{case 1: } n ,& \text{if } r\leq R/\sqrt{n},\\
    \text{case 2: } \frac{R^2}{r^2}\log\frac{enr^2}{R^2} ,& \text{if } R/\sqrt{n}< r \leq R,
  \end{cases}
\end{align*}
which corresponds to the following regimes of the critical local metric entropy radius $ \delta _\varepsilon ^* $:
\begin{align*}
  \begin{cases}
    \text{case 1: } \delta _\varepsilon ^{*2} \asymp \frac{n}{N(1-2\varepsilon )}\sigma ^2 ,& \text{if }\delta_\varepsilon ^* \leq \frac{R}{\sqrt{n}};\\
    \text{case 2: } \frac{\log (en\delta _\varepsilon ^{*2}/R^2)}{(en\delta _\varepsilon ^{*2}/R^2)^2} \asymp \frac{N(1-2\varepsilon )R^2}{\sigma ^2}\frac{1}{n^2},& \text{if }  \frac{R}{\sqrt{n}}< \delta _\varepsilon ^* < R.
  \end{cases} 
\end{align*}
Solving the above equations gives
\begin{align*}
  \delta _\varepsilon ^{*2}\wedge R^2 \asymp \begin{cases}
    R^2\sqrt{\frac{\log (en/\sqrt{\lambda })}{\lambda }},& \log en < \lambda  < n^2 \\
    R^2\frac{n}{\lambda } ,& \lambda  \geq n^2.
  \end{cases} 
\end{align*}
where $ \lambda := \frac{N(1-2\varepsilon )R^2}{\sigma ^2} $. In the edge case when $ \lambda \leq \log en $, we have
\begin{align*}
  \log (2n) \geq \log(en)/4 \geq \lambda /4 = \frac{N(1-2\varepsilon ) (R/2)^2}{\sigma ^2}. 
\end{align*}
On the other hand, since $ \{\pm(R/2)\hat{e}_i:\,i\in[n]\} \subset B_1(R)\cap B(0,R/2) $, meaning that $ \log\mloc(\frac{R}{2},c)\geq \log 2n $ as long as $ c $ is large enough. Thus we have $ \delta _\varepsilon ^* \geq R/2 $, and as a result $ \delta _\varepsilon ^{*2}\wedge R^2 \asymp R^2 $ in this case. This can be absorbed into the first case above, giving the final result.

\section{Proof of Lower Bound}\label{appendix:lower_bound_proof}

\begin{proof}[Proof of \autoref{lemma:effective_sample_size_reduction}]
  Let $ N_0:= N-2k $. In the following we construct the adversarial corruption algorithm $ \mathcal{C}_\pi $.

  The corruption algorithm does the following: First we keep $ N_0 $ clean anchor data $ Z:= \tilde{X}_{1:N_0} $, then compute the posterior distribution of $ V $ given $ \tilde{X}_{1:N_0} $ as
  \begin{align*}
    \pi_{Z}(v) :=& \mathbbm{P}_{  }\left( V=v | Z \right)= \frac{ \pi(v)\prod_{i=1}^{N_0}f_{\mu_v}(\tilde{X}_i) }{ \sum_{v'\in [M]}\pi(v')\prod_{i=1}^{N_0}f_{\mu_{v'}}(\tilde{X}_i) },\quad v\in [M],
  \end{align*}
  where $ f_{\mu _v} = \phi (\, \cdot \, ,\mu _v,\sigma ^2I_n) $ is the density function of $ \mathcal{N}(\mu _v,\sigma ^2I_n) $. Then we draw $ V' \sim \pi_{Z} $ and consequently we have
  \begin{align*}
    V,V' |Z\mathop{ \sim }\limits^{i.i.d.} \pi_{Z}.  
  \end{align*}

  On the remaining $ 2k $ samples, we pair them to form $ k $ pairs, and each pair $ (\tilde{X}_j,\tilde{X}_{j+1}) $ has distribution $ P_{\mu _V}\otimes P_{\mu _{V}} $
   (given $ V $), $ j= N_0+1,N_0+3,\ldots, N-1 $. The adversarial algorithm randomly replaces one of the two samples in each pair with a sample drawn from $ P_{\mu _{V'}} $. Thus each corrupted pair has distribution 
  \begin{align*}
    (X_j,X_{j+1}) | V,V'\sim \frac{1}{2}\big( P_{\mu _V}\otimes P_{\mu _{V'}} + P_{\mu _{V'}}\otimes P_{\mu _{V}} \big) := S_{V,V'}.
  \end{align*}
  As a result, within the $ 2k $ samples, there are $ k $ samples corrupted, while the other $ k $ samples remain clean. Putting them together we denote $ Y \in \mathbbm{R}^{2k\times n} $. Then the data after corruption is $ X_{1:N} = (Z,Y) $.
  Also, we remark that since $ V,V' $ are i.i.d. given $ Z $, the distribution of $ Y $ is symmetric in $ V $ and $ V' $: $ S_{v,v'} = S_{v',v} $ for any $ v,v'\in [M] $.
  As a result,
  \begin{align}\label{equation:effective_sample_size_reduction_equation_1}
    (V,V',Y) |Z\mathop{ = }\limits^{\mathrm{d}} (V',V,Y) |Z.
  \end{align}

  Putting the above together we have for any estimator $ \hat{\mu }: \mathbbm{R}^{N\times n}\to \mathbbm{R}^n $,
  \begin{align*}
    \mathbbm{E}_{ \tilde{X},\mathcal{C} }\left[ \left\Vert \hat{\mu}(X_{1:N})- \mu_V  \right\Vert ^2 \right] =& \mathbbm{E}_Z \left[ \mathbbm{E}_{ V,V',Y|Z }\left[ \left\Vert \hat{\mu}(Z,Y)- \mu_V  \right\Vert ^2 | Z  \right] \right]\\
    \mathop{ = }\limits^{(i)} & \mathbbm{E}_Z \left[ \mathbbm{E}_{ V,V',Y|Z }\left[ \left\Vert \hat{\mu}(Z,Y)- \mu_{V'}  \right\Vert ^2 | Z  \right] \right],
  \end{align*}
  where $ (i) $ by \autoref{equation:effective_sample_size_reduction_equation_1}. Rearranging and using the parallelogram identity we have
  \begin{align*}
    \mathbbm{E}_{ \tilde{X},\mathcal{C} }\left[ \left\Vert \hat{\mu}(X_{1:N})- \mu_V  \right\Vert ^2 \right] =& \frac{1}{2} \mathbbm{E}_{ \tilde{X},\mathcal{C} }\left[ \left\Vert \hat{\mu}(X_{1:N})- \mu_V  \right\Vert ^2 + \left\Vert \hat{\mu}(X_{1:N})- \mu_{V'}  \right\Vert ^2 \right] \\
    \geq & \frac{1}{4} \mathbbm{E}_{ \tilde{X},\mathcal{C} }\left[ \left\Vert \mu_V - \mu_{V'}  \right\Vert ^2 \right].\\
  \end{align*}
  For the right-hand side we have
  \begin{align*}
     \mathbbm{E}_{ V,V'|Z }\left[ \left\Vert \mu_V - \mu_{V'}  \right\Vert ^2 | Z  \right] =& \mathbbm{E}_{ V,V'|Z }\left[ \left\Vert \mu_V - \mathbbm{E}_{ V|Z }\left[ \mu_V | Z  \right] + \mathbbm{E}_{ V|Z }\left[ \mu_V | Z  \right] - \mu_{V'}  \right\Vert ^2 \big| Z  \right]\\
    \mathop{ = }\limits^{(i)} & 2\mathbbm{E}_{ V|Z }\left[ \left\Vert \mu_V - \mathbbm{E}_{ V|Z }\left[ \mu_V | Z  \right]  \right\Vert ^2 \big| Z  \right] \\
    &+ \underbrace{2\mathbbm{E}_{ V,V'|Z }\left[ \left\langle \mu_V - \mathbbm{E}_{ V|Z }\left[ \mu_V | Z  \right], \mathbbm{E}_{ V|Z }\left[ \mu_V | Z  \right] - \mu_{V'}  \right\rangle \big| Z  \right]}_{=0}
  \end{align*}
  where $ (i) $ since $ V,V' $ are i.i.d. given $ Z $. Taking expectation over $ Z $ we have
  \begin{align}\label{equation:effective_sample_size_reduction_eqa_1}
     \mathbbm{E}_{ \tilde{X},\mathcal{C} }\left[ \left\Vert \hat{\mu}(X_{1:N})- \mu_V  \right\Vert ^2 \right] \geq & \frac{1}{2}\mathbbm{E}_{ Z }\mathbbm{E}_{ V|Z }\left[ \left\Vert \mu_V - \mathbbm{E}_{ V|Z }\left[ \mu_V | Z  \right]  \right\Vert ^2 \big| Z  \right] .
  \end{align}
  For the Bayesian risk of clean data of size $ N_0 $, we note that $ R_\mathrm{ clean }(N_0,\pi) $ takes minimum when the estimator is the posterior mean $ \mathbbm{E}_{ V|Z }\left[ \mu_V | Z  \right] $. Thus we have
  \begin{align*}
    R_\mathrm{ clean }(N_0,\pi) =& \mathbbm{E}_{ Z }\mathbbm{E}_{ V|Z }\left[ \left\Vert \mu_V - \mathbbm{E}_{ V|Z }\left[ \mu_V | Z  \right]  \right\Vert ^2 \big| Z  \right] ,  
  \end{align*}
  where the first argument $ N_0 = N-2k $ is the sample size of $ Z $, the clean fraction we keep in the corruption. Substituting this into \eqref{equation:effective_sample_size_reduction_eqa_1} yields the desired result.
  
\end{proof}

\begin{lemma}[Local Packing Lower Bound]\label{lemma:local_packing_lower_bound}
  Let $ c $ be the universal constant used in the local metric entropy, fixed sufficiently large as mentioned in \autoref{appendix:tournament_on_tree_framework}. Then for any $ \eta $ such that $ \log M^\mathrm{ loc }(\eta,c) > 4(\frac{N-2k}{2\sigma ^2}\eta^2 \vee \log 2) $, we have
  \begin{align*}
    \mathfrak{M}\geq \eta^2 / 16c^2.
  \end{align*}
\end{lemma}

\begin{proof}
  The proof follows from Fano's method and then using \autoref{lemma:effective_sample_size_reduction}. 
  
  Let notations be as in \autoref{lemma:effective_sample_size_reduction}. Set $ \mu _1,\ldots, \mu _M $ be the local packing set satisfying $ M=M^\mathrm{ loc }(\eta ,c) $ and $ \pi $ be the uniform prior on $ [M] $. Then by Markov's inequality and following \cite{wainwrightHighDimensionalStatisticsNonAsymptotic2019}[Proposition 15.1, Equation 15.31], we have
  \begin{align*}
    R_\mathrm{ clean } (N_0,\pi) \geq \frac{\eta ^2}{4c^2}\Big\{ 1-\frac{I(\tilde{X}_{1:N_0}; V) + \log 2}{\log M}   \Big\},
  \end{align*}
  where $ I(\tilde{X}_{1:N_0}; V) $ is the mutual information between $ \tilde{X}_{1:N_0} $ and $ V $. For Gaussian distribution $ \mathcal{N}(\mu _v,\sigma ^2I_n) $, we have
  \begin{align*}
    I(\tilde{X}_{1:N_0}; V) \leq &\frac{1}{M}\sum_{v=1}^{M}D_\mathrm{ KL }(P_{\mu _v}^{\otimes N_0} \Vert Q^{\otimes N_0}) \\
    =& \frac{N_0}{2\sigma ^2M}\sum_{v=1}^{M}\left\Vert \mu _v - \bar{\mu}  \right\Vert ^2 \\
    \leq& \frac{N_0}{2\sigma ^2}\cdot \eta ^2
  \end{align*}
  using \cite{prasadan_information_2026}[Lemma A.2], where $ Q = \mathcal{N}(\bar{\mu},\sigma ^2I_n) $ and $ \bar{\mu} = \frac{1}{M}\sum_{v=1}^{M}\mu _v $. Putting the above together, we have
  \begin{align*}
    \log M  > 4(\frac{N-2k}{2\sigma ^2}\eta^2 \vee \log 2) \geq 2(\frac{N-2k}{2\sigma ^2}\eta^2 + \log 2 ) \geq 2(I(\tilde{X}_{1:N_0}; V) + \log 2).
  \end{align*}
  Then substituting this into the lower bound of $ R_\mathrm{ clean }(N_0,\pi) $, and applying \autoref{lemma:effective_sample_size_reduction} we have
  \begin{align*}
    \mathfrak{M} \geq \mathop{ \inf }\limits_{\hat{\mu }}R_\mathrm{ corrupt }(\hat{\mu }, \mathcal{C} , \pi ) \geq \frac{1}{2} R_\mathrm{ clean }(N-2k,\pi) \geq \frac{\eta ^2}{16c^2}. 
  \end{align*}  
  
\end{proof}

The following characterizes the maximum separation of two Gaussian distributions given a total variation distance constraint. It is a direct result from \cite{kotekalOptimalEstimationNull2025}[Appendix E, Proposition 8], with a proper renotation adjustment. 

\begin{lemma}[Maximum Separation of Two Gaussians]\label{lemma:maximum_separation_two_gaussians}
  Let $ \eta \in (1/4,1/2) $ and $ \sigma > 0 $. Let
  \begin{align*}
    \omega(\eta):= \sup\big\{ \dfrac{ \left\vert \theta _1-\theta _2 \right\vert  }{ \sigma  } : d_\mathrm{ TV }( \mathcal{N}(\theta _1,\sigma ^2), \mathcal{N}(\theta _2,\sigma ^2)) \leq 2\eta  \big\}. 
  \end{align*}
  Then we have
  \begin{align*}
    \omega (\eta) \geq \sqrt{ 2 \log \left( \dfrac{ 1 }{ 2(1-2\eta) }  \right) } .
  \end{align*}
\end{lemma}

\begin{corollary}\label{corollary:maximum_separation_two_gaussians}
  Let $ \eta \in (1/4,1/2) $ and $ \sigma > 0 $. Then for any $ \theta _1,\theta_ 2\in \mathbbm{R}^n $ such that $ \left\Vert \theta _1-\theta _2 \right\Vert ^2 \leq \sigma ^2 \vee 2\sigma ^2\log \left( \dfrac{ 1 }{ 2(1-2\eta) }  \right)$, we have
  \begin{align*}
    d_\mathrm{ TV }( \mathcal{N}(\theta _1,\sigma ^2I_n), \mathcal{N}(\theta _2,\sigma ^2I_n)) \leq 2\eta.
  \end{align*}
  We also note that $ \sigma ^2 \vee 2\sigma ^2\log \left( \dfrac{ 1 }{ 2(1-2\eta) }  \right) \asymp \sigma ^2\log \left( \dfrac{ 1 }{ 1-2\eta }  \right) $ for $ \eta \in (1/4,1/2) $.

\end{corollary}

\begin{proof}
  By the rotational invariance of the total variation distance and the tensorization property that
  \begin{align*}
    d_{\mathrm{ TV } }(P\times S, Q\times S) = d_{\mathrm{ TV } }(P,Q) ,
  \end{align*}
  the preceding lemma can be extended to the multivariate case:
  \begin{align*}
    \left\Vert \theta _1-\theta _2 \right\Vert ^2 \leq 2\sigma ^2\log \left( \dfrac{ 1 }{ 2(1-2\eta) }  \right).  
  \end{align*} 
  
  On the other hand, by Pinsker's inequality and the KL divergence formula for Gaussian distributions, we have
  \begin{align*}
     d_{\mathrm{ TV } }( \mathcal{N}(\theta _1,\sigma ^2I_n), \mathcal{N}(\theta _2,\sigma ^2I_n)) \leq & \sqrt{\frac{1}{2}D_\mathrm{ KL }( \mathcal{N}(\theta _1,\sigma ^2I_n) \Vert \mathcal{N}(\theta _2,\sigma ^2I_n))} \\
    = & \frac{ \left\Vert \theta _1-\theta _2 \right\Vert }{ 2\sigma },
  \end{align*}
  meaning that if $ \left\Vert \theta _1-\theta _2 \right\Vert ^2 \leq \sigma ^2 $, then $ d_{\mathrm{ TV } }( \mathcal{N}(\theta _1,\sigma ^2I_n), \mathcal{N}(\theta _2,\sigma ^2I_n)) \leq 1/2 < 2\eta $ for $ \eta \in (1/4,1/2) $.

  Combining the two results gives the desired inequality.
\end{proof}

\begin{proof}[Proof of \autoref{lemma:contamination_lower_bound_adversary_construction}]
  For the two hypotheses $ b\in \{0,1\} $ which corresponds to $ \mu _0 $ and $ \mu _1 $ respectively, denote their gaussian distributions as $ P_b $. Let the average distribution be $ \bar{P} = \frac{1}{2}(P_0+P_1) $. We have
  \begin{align*}
    d_\mathrm{ TV }( P_0, \bar{P}) = d_\mathrm{ TV }( P_1, \bar{P}) = \frac{1}{2}d_\mathrm{ TV }( P_0, P_1) \leq \varepsilon _-.  
  \end{align*}
  
  The adversarial algorithm $ \mathcal{C}_b $ is as follows: For each $ \tilde{X}_i \sim P_b $, we construct its maximal coupling $ (\tilde{X}_i, Y_i) $ with $ \bar{P} $ with the following property:
  \begin{align*}
     \tilde{X}_i\sim P_b,\quad Y_i\sim \bar{P},\quad \mathbbm{P}_{  }\left( \tilde{X}_i\neq Y_i  \right) = d_\mathrm{ TV }( P_b, \bar{P}) \leq \varepsilon _-.
  \end{align*}
  This can be done by accept-reject sampling. Then we count the number of disagreements between $ \tilde{X}_i $ and $ Y_i $, denoted
  \begin{align*}
    B_i:= \mathbbm{1}(\tilde{X}_i \neq Y_i) \sim \mathrm{ Bernoulli }( d_\mathrm{ TV }( P_b, \bar{P}) ) ,
  \end{align*}
  and let $ B=\sum_{i=1}^{N}B_i $. Construct the couplings independently for each $ i\in [N] $, then $ B \sim \mathrm{ Binom }(N,d_\mathrm{ TV }( P_b, \bar{P})) $.

  The adversarial algorithm $ \mathcal{C}_b $ does the following:
  \begin{itemize}[topsep=2pt,itemsep=0pt]
    \item If $ B \leq k $, replace all the $ \tilde{X}_i $ with $ Y_i $;
    \item If $ B > k $, do nothing.
  \end{itemize}
  We remark that in the first case, only at most $ k $ samples are replaced, which is within the corruption budget. The resulting distribution of the corrupted data $ X_{1:N} $ is denoted $ \tilde{P}_b^{(N)} $. Then since $ B $ is Binomially distributed, we have
  \begin{align*}
    \mathbbm{P}_{  }\left( B>k \right) \mathop{ \leq }\limits^{(i)} & \mathbbm{P}_{  }\left( \mathrm{ Binom }(N, \varepsilon _-) > k \right) \\
    =& \mathbbm{P}_{  }\left( \mathrm{ Binom }(N, \varepsilon _-) > N\varepsilon _- + a\sqrt{N} \right) \\
    \mathop{ \leq }\limits^{(ii)} &\exp\left[ -2a^2 \right]\\
    \mathop{ < }\limits^{(iii)}& 1/8,
  \end{align*}
  where $ (i) $ by the stochastic dominance $ B \sim \mathrm{Binom}(N,d_\mathrm{ TV }( P_b, \bar{P})) \leq_{\mathrm{st}} \mathrm{ Binom }(N,\varepsilon _-)  $; $ (ii) $ using Hoeffding's inequality; and $ (iii) $ by choosing $ a $ large enough. By the above construction, we have
  \begin{align*}
    d_\mathrm{ TV }( \tilde{P}^{(N)}_b , \bar{P}^{\otimes N}) \leq \mathbbm{P}_{  }\left( X_{1:N} \neq Y_{1:N} \right) = \mathbbm{P}_{  }\left( B>k \right) < 1/4.
  \end{align*}
  Then by triangle inequality we have
  \begin{align*}
     d_\mathrm{ TV }( \tilde{P}^{(N)}_0, \tilde{P}^{(N)}_1) \leq d_\mathrm{ TV }( \tilde{P}^{(N)}_0, \bar{P}^{\otimes N}) + d_\mathrm{ TV }( \tilde{P}^{(N)}_1, \bar{P}^{\otimes N}) < 1/2 < 1.
  \end{align*}
  Then using Le Cam's two point method \cite{wainwrightHighDimensionalStatisticsNonAsymptotic2019}[Equation 15.14], we have the desired result.
  
\end{proof}

\begin{proof}[Proof of \autoref{lemma:contamination_lower_bound}]
  Let $ a $ be the universal constant from \autoref{lemma:contamination_lower_bound_adversary_construction}. We first deal with the edge case of $ \varepsilon _- = \varepsilon - a/\sqrt{N} \leq 1/4 $. In this case we have that $ N $ is upper bounded such that $ N \leq 64a^2 $ since $ \varepsilon > 3/8 $ by assumption. Subsequently, since $1-2\varepsilon \geq \frac{1}{N}>\frac{1}{64a^2}$, we have $ \varepsilon  $ strictly bounded away from $ 1/2 $. Applying \autoref{proposition:agreement_with_existing_result} with $ \kappa = 1/(128a^2) $ gives the desired result. Now we turn to the $ \varepsilon _- > 1/4 $ side where we can apply \autoref{lemma:contamination_lower_bound_adversary_construction}.

  We first consider the case of bounded $ K $.  Let $ \varepsilon _- $ defined as in \autoref{lemma:contamination_lower_bound_adversary_construction}.
  By \autoref{lemma:useful_properties}[(iii)] and \autoref{corollary:maximum_separation_two_gaussians}, there exists $ \mu _0,\mu _1\in K $ such that
    \begin{align}\label{equation:contamination_lower_bound_eqa_1}
      \left\Vert \mu _0-\mu _1 \right\Vert ^2 = \big( \sigma ^2 \vee 2\sigma ^2\log\frac{1}{2(1-2\varepsilon_-) }\big) \wedge (d/3)^2 .
    \end{align}
    Consequently, we have $ d_\mathrm{ TV }( \mathcal{N}(\mu _0,\sigma ^2I_n), \mathcal{N}(\mu _1,\sigma ^2I_n)) \leq 2\varepsilon _- $ and that
    \begin{align*}
      \left\Vert \mu _0-\mu _1 \right\Vert ^2 \asymp \sigma ^2\log\frac{1}{1-2\varepsilon_- } \wedge d^2.
    \end{align*}
    Then by applying \autoref{lemma:contamination_lower_bound_adversary_construction}, we have
    \begin{align*}
      \mathfrak{M} \gtrsim \left\Vert \mu _0-\mu _1 \right\Vert ^2 \asymp \sigma ^2\log\frac{1}{1-2\varepsilon_- } \wedge d^2.
    \end{align*}
    It suffices to prove that $ \log\frac{1}{1-2\varepsilon_- } \asymp \log\frac{1}{1-2\varepsilon } $.

    Notice that by construction we have $ \varepsilon = \frac{k}{N} $ where $ k< \frac{N}{2} $. As a result, we have
    \begin{align*}
      1-2\varepsilon  = \frac{N-2k}{N} \geq \frac{1}{N} .
    \end{align*}
    Now discuss two cases:
    \begin{itemize}[topsep=2pt,itemsep=0pt]
      \item If $ 1-2\varepsilon \geq 2a/\sqrt{N} $, then we have
      \begin{align*}
        1-2\varepsilon < 1-2\varepsilon _- = (1-2\varepsilon ) + \frac{2a}{\sqrt{N}} \leq 2(1-2\varepsilon ) 
      \end{align*}
      as a result we have $ \log\frac{1}{1-2\varepsilon_- } \asymp \log\frac{1}{1-2\varepsilon } $ (using $\varepsilon > 3/8$). 
      \item If $ 1/N \leq 1-2\varepsilon < 2a/\sqrt{N} $, then we have
      \begin{align*}
        \log \frac{1}{1-2\varepsilon_- } \asymp \log N \asymp \log\frac{1}{1-2\varepsilon }.
      \end{align*}      
    \end{itemize}
    Then noticing that $ \varepsilon ^2\asymp 1 $ in this case, we can conclude the proof. The case of unbounded $ K $ follows similarly by choosing $ \left\Vert \mu _0-\mu _1 \right\Vert ^2 \asymp \sigma ^2\log\frac{1}{1-2\varepsilon_- }  $ in \autoref{equation:contamination_lower_bound_eqa_1}.

\end{proof}

\begin{proof}[Proof of \autoref{theorem:minimax_lower_bound}]
  We prove for the bounded $ K $ case, and the unbounded $ K $ case follows similarly. According to \autoref{proposition:agreement_with_existing_result}, for some universal constant $ 0<\kappa <1/2 $, if $ \varepsilon \leq 1/2-\kappa $, we have
  \begin{align*}
    \mathfrak{M}\gtrsim  \max\{\delta _\varepsilon^{*2} , \sigma ^2\varepsilon ^2\log(e/(1-2\varepsilon )) \} \wedge d^2.
  \end{align*}
  Then it suffices to consider the case when $ 1/2-\kappa < \varepsilon < 1/2  $, where we without loss of generality set $ \kappa = 1/8 $ and thus $ 3/8 < \varepsilon <1/2 $.

  The discussion is split depending on the value of $ \delta _\varepsilon ^* $.
  \begin{itemize}[topsep=2pt,itemsep=0pt]
    \item \textbf{If $ \delta _\varepsilon ^* = 0 $}: We argue that $ K $ must be a singleton set. If not, then using \autoref{lemma:useful_properties}[(iii)], by choosing $ r  $ sufficiently small, we have $ \mloc(r ,c) \geq 2 $ providing that $ c $ is sufficiently large. However, $ (1-2\varepsilon )N\frac{r^2}{\sigma ^2} \xrightarrow[]{r\to 0} 0  $. Choosing $ r>0 $ sufficiently small gives
    \begin{align*}
      (1-2\varepsilon )N \frac{r^2}{\sigma ^2} < \log 2 \leq \log \mloc(r ,c).
    \end{align*}
    As a result, $ \delta _\varepsilon ^* \geq r>0 $, yielding a contradiction. Thus $ K $ is a singleton set, and the minimax risk is trivial.
       
    \item \textbf{If $ \frac{N-2k}{\sigma ^2}\delta_\varepsilon ^{*2} > 8\log 2 $ (large $ \delta _\varepsilon ^* $)}: We argue that $ \frac{\delta _\varepsilon ^*}{4} $ satisfies the condition in \autoref{lemma:local_packing_lower_bound}.
    
    Since any $ 0<r<\delta _\varepsilon ^* $ satisfies $ \log\mloc(r,c) > (1-2\varepsilon )N\frac{r^2}{\sigma ^2} $, by \autoref{lemma:useful_properties}[(i)] we have
    \begin{align*}
      \log\mloc(\frac{\delta _\varepsilon ^*}{4},c)\geq &\log\mathop{ \lim }\limits_{s\uparrow \delta _\varepsilon ^*} \mloc(s,c) \geq  \mathop{ \lim }\limits_{s\uparrow \delta _\varepsilon ^*} \frac{N-2k}{2\sigma ^2}s^2\\
      =& \frac{N-2k}{2\sigma ^2}\delta_\varepsilon ^{*2} > 4\Big[ \frac{N-2k}{2\sigma ^2}(\delta_\varepsilon ^{*}/4)^2 \vee \log 2 \Big].
    \end{align*}
    Then applying \autoref{lemma:local_packing_lower_bound} we have
    \begin{align*}
      \mathfrak{M} \geq & \frac{ (\delta _\varepsilon ^*/4)^2}{16c^2} \asymp \delta _\varepsilon ^{*2} \gtrsim  \delta _\varepsilon ^{*2}\wedge d^2.
    \end{align*} 
    Together with \autoref{lemma:contamination_lower_bound} we have
    \begin{align*}
       \mathfrak{M} \gtrsim & (\delta _\varepsilon ^{*2} \vee \sigma ^2\varepsilon ^2\log(1/(1-2\varepsilon ))) \wedge d^2 .
    \end{align*}

    \item \textbf{If $ \frac{N-2k}{\sigma ^2}\delta_\varepsilon ^{*2} \leq 8\log 2 $ (small $ \delta _\varepsilon ^* $)}: By the same argument as in \cite{prasadan_information_2026}[proof of Theorem 3.10, case 3], we have $ d\leq 12\delta _\varepsilon ^{*} $ in this case. Now we argue that $ d/6 $ satisfies the condition in \autoref{lemma:local_packing_lower_bound}.
    
    Choosing $ c $ sufficiently large such that we have
    \begin{align*}
      \log\mloc(\frac{d}{6}, c) > 64\log 2 \geq 4\big( 2\frac{N-2k}{\sigma ^2}(\delta _\varepsilon ^{*})^2 \big) \vee 4\log 2 \geq 4\big( \frac{N-2k}{2\sigma ^2}(d/6)^2 \vee \log 2 \big).
    \end{align*}
    Thus we can apply \autoref{lemma:local_packing_lower_bound} to get
    \begin{align*}
      \mathfrak{M} \gtrsim d^2  \gtrsim \delta _\varepsilon ^{*2}\wedge d^2.
    \end{align*}
    Together with \autoref{lemma:contamination_lower_bound} we have
    \begin{align*}
       \mathfrak{M} \gtrsim  & (\delta _\varepsilon ^{*2} \vee \sigma ^2\varepsilon ^2\log(1/(1-2\varepsilon ))) \wedge d^2 .
    \end{align*}
  
  \end{itemize}

\end{proof}

\section{Constant Calibration Near the Breakdown Point}\label{appendix:constant_calibration_near_breakdown_point}

In \cite{prasadan_information_2026}, their constant setting was done in the regime when $ \varepsilon $ was strictly bounded away from $ 1/2 $. As a result, the dependence of the constants on $ \varepsilon $ with proximity to the breakdown point $ 1/2 $ was absorbed. In this section, we focus on the constants in the robust testing procedure in \cite{prasadan_information_2026}[Theorem 3.6] in the near-breakdown regime $ \varepsilon \in [1/4,1/2) $, to make the dependence on $ \varepsilon $ explicit. This robust testing lies in the core of the robust estimation procedure, as we mentioned in \autoref{appendix:tournament_on_tree_framework}.

We begin by introducing some notations and how they compare to the original definitions in \cite{prasadan_information_2026}[Appendix B]. Throughout the proof necessary constants from the original proof are added subscript $ \varepsilon $ to indicate their dependence on $ \varepsilon $. Specifically, the notation $ C,C',\rho $ remains unchanged since they are universal constants, while we use $ \alpha _\varepsilon  $, $ \beta _\varepsilon  $, $ L_\varepsilon  $, etc., in place of their original definitions. We also write
\begin{align*}
  a_\varepsilon := 1-2\varepsilon ,\quad z_\varepsilon := \log \frac{e}{a_\varepsilon } = \log \frac{e}{1-2\varepsilon },
\end{align*}
for notational convenience. Additionally, we recall the definitions $ g:[0,1/2)\mapsto \mathbbm{R} $ and $ h:[0,1/2)\mapsto \mathbbm{R} $ that
\begin{align*}
  g(t):=& (\frac{1}{2}+t)\log (\frac{1}{2} + t) + (\frac{1}{2}-t)\log ( 1-2t ),\\
  h(t):=& t\big( 1 - \exp\left[ \dfrac{ 4g(t) }{ 1-2t }  \right] \big).
\end{align*}

Now we first assess the endpoint behavior of $ g $ and $ h $ as $ t\to 1/2 $. Then we proceed to the proof of tail bounds of the robust testing procedure.
\begin{lemma}\label{equation:function_g_h_endpoint_behavior_1}
  There exist universal constants $ Q_1,Q_2 $ such that for every $ t \in [1/4,1/2) $, 
  \begin{align}
    Q_1 \log\dfrac{ e }{ 1-2t  } \leq \dfrac{ -4g(t) }{ 1-2t } \leq Q_2 \log\dfrac{ e }{ 1-2t  },   
  \end{align}
  and
  \begin{align}\label{equation:function_g_h_endpoint_behavior_2}
    Q_1(1-2t) \leq 1-2h(t) \leq Q_2(1-2t).
  \end{align}
\end{lemma}

\begin{proof}
  Set $ s = 1-2t \in (0,1/2] $. Then we have
  \begin{align*}
    -\dfrac{ 4g(t) }{ 1-2t } =& -\dfrac{ 4\big( (1-s/2)\log(1-s/2) + (s/2)\log(s) \big) }{ s } \\
    =& 2\log(1/s) - 4\underbrace{\frac{ 1-s/2 }{ s }\log(1-s/2)}_{O(1)\text{ on } s\in (0,1/2]} \\
    \asymp & \log(e/s),
  \end{align*}
  proving the first inequality. Also, it gives $ \exp(4g(t)/(1-2t)) \asymp s^2 $, which implies that
  \begin{align*}
    1-2h(t) = s + (1-s)\exp(4g(t)/(1-2t)) \asymp s,
  \end{align*}
  proving the second inequality.  
\end{proof}

\begin{lemma}[Robust Pairwise Testing Near Breakdown Point]\label{lemma:robust_pairwise_testing_near_breakdown_point}
  Fix a sufficiently large constant $ C>2 $. There exist absolute positive constants $ D_1,Q_0 >0 $, independent of $ \varepsilon $, with the following property: For every $ \varepsilon = \frac{k}{N} \in [1/4,1/2) $, consider testing $ H_0 : \mu \in B(\nu _1,r) $ versus $ H_1 : \mu \in B(\nu _2,r) $ for $ \left\Vert \nu _1-\nu _2 \right\Vert \geq Cr $ with $ \nu _1,\nu _2\in K $. If $ r \geq D_1\varepsilon \sigma \sqrt{z_\varepsilon } $, then the robust testing procedure $ \psi $ in \cite{prasadan_information_2026}[Theorem 3.6] satisfies
  \begin{align*}
    \mathop{ \sup }\limits_{\left\Vert \mu -\nu _1 \right\Vert \leq r} \mathbbm{P}_{ \mu ,\mathcal{C} }\left( \psi = 1 \right) \vee  \mathop{ \sup }\limits_{\left\Vert \mu -\nu _2 \right\Vert \leq r } \mathbbm{P}_{ \mu ,\mathcal{C} }\left( \psi = 0 \right)  \leq \exp\left[ -Q_0a_\varepsilon N\frac{r^2}{\sigma ^2} \right],
  \end{align*}
  for any admissible adversary $ \mathcal{C} $ with corruption budget $ k $.
\end{lemma}

\begin{proof}
  Let $ Q_1,Q_2 $ be the constants from the previous lemma. 
  Set
  \begin{align*}
    C':=\dfrac{ C-2 }{ 2\sqrt{2\pi} },\quad \varrho := \exp\left[ -(C')^2r^2/\sigma ^2 \right],\quad \alpha _\varepsilon := \frac{1}{2} - Q_3a_\varepsilon ,  
  \end{align*}
  where $ Q_3 \in (0,1/(4Q_2)) $ is some constant chosen small enough such that $ \alpha _\varepsilon \in [1/4,1/2) $ for every $ \varepsilon \in [1/4,1/2) $. Then directly we have
  \begin{align}\label{equation:robust_testing_constant_condition_3}
    h(\alpha _\varepsilon ) - \varepsilon  \mathop{ = }\limits^{(i)}  \frac{a_\varepsilon }{2} - \big(\frac{1}{2} - h(\alpha_\varepsilon )\big) 
    \mathop{ \geq }\limits^{(ii)}  a_\varepsilon \big( \frac{1}{2} - Q_2Q_3 \big)
    > 0,
  \end{align}
  where $ (i) $ from the definition of $ a_\varepsilon  $; $ (ii) $ using \autoref{equation:function_g_h_endpoint_behavior_2}. This verifies a counterpart of \cite{prasadan_information_2026}[Appendix B, Equation (B.1)]. Then we set 
  \begin{align*}
    L_\varepsilon := \frac{3}{2}\sqrt{-\dfrac{ 4g(\alpha _\varepsilon ) }{ 1-2\alpha _\varepsilon  } } ,\quad \beta _\varepsilon := C'\big( \Phi(L_\varepsilon )-\frac{1}{2} \big),
  \end{align*}
  then directly we have 
  \begin{align}\label{equation:robust_testing_constant_condition_0}
    \Phi^{-1}\big( \frac{1}{2} + \frac{\beta _\varepsilon}{C'} \big) = L_\varepsilon > \sqrt{-\dfrac{ 4g(\alpha _\varepsilon ) }{ 1-2\alpha _\varepsilon  } },
  \end{align}
  which is a counterpart of their Equation (B.2). We further set 
  \begin{align*}
    C_{1,\varepsilon } := \frac{2L_\varepsilon }{\beta _\varepsilon } ,\quad C_{2,\varepsilon } := \frac{\beta _\varepsilon }{L_\varepsilon },
  \end{align*}
  which follow the settings of their (B.3) and (B.4), and satisfy $ C_{1,\varepsilon }C_{2,\varepsilon }=2 $. We note that we have $ \beta _\varepsilon \asymp 1 $ since \autoref{equation:function_g_h_endpoint_behavior_1} gives $ -4g(\alpha _\varepsilon )/(1-2\alpha _\varepsilon ) $ is bounded from below. As a result, we have $ C_{1,\varepsilon }^2\asymp z_\varepsilon  $, and then
  \begin{align*}
    C_{3,\varepsilon } := \frac{C_{2,\varepsilon }^2}{2} \wedge \frac{(1/2-\alpha _\varepsilon )C'^2}{2} \asymp a_\varepsilon , 
  \end{align*}
  where the final comparison follows since $ a_\varepsilon z_\varepsilon \lesssim 1 $. Now we finish the calibration of constants as counterparts of (B.1)-(B.4) in \cite{prasadan_information_2026}[Appendix B]. 
  
  Then we verify that they satisfy the two constant conditions (i)-(ii) from their proof. Choose $ D_1 $ sufficiently large such that
  \begin{align*}
    D_1 > \max\big\{ \mathop{ \sup }\limits_{\varepsilon } \dfrac{ C_{1,\varepsilon } }{ \sqrt{z_\varepsilon } } , \mathop{ \sup }\limits_{\varepsilon } \dfrac{ 4 }{ C' }\dfrac{ L_\varepsilon  }{ \sqrt{z_\varepsilon } } \big\} < \infty.
  \end{align*}
  Here both suprema are finite by the previous discussion. As a result, if $ r/\sigma  \geq D_1\varepsilon \sqrt{z_\varepsilon } $, then using the first term we have
  \begin{align}\label{equation:robust_testing_constant_condition_1}
    \frac{r}{\sigma } \geq  C_{1,\varepsilon }\varepsilon .
  \end{align}
  Also, since $ \varepsilon \geq 1/4 $, we have for the second term
  \begin{align}\label{equation:robust_testing_constant_condition_2}
    \frac{C'r}{\sigma } \geq C'D_1\varepsilon \sqrt{z_\varepsilon } \geq \dfrac{ C'D_1 }{ 4 }\sqrt{z_\varepsilon} > L_\varepsilon .
  \end{align}
  This verifies the two constant conditions (i)-(ii) in \cite{prasadan_information_2026}[Appendix B]:
  \begin{itemize}[topsep=2pt,itemsep=0pt]
    \item[(i)] Verify that $ \varepsilon < \alpha _\varepsilon (1-\varrho) $:
    \begin{align*}
      \alpha _\varepsilon (1-\varrho) =& \alpha _\varepsilon \big( 1 - \exp\left[ -(C')^2r^2/\sigma ^2 \right] \big) \\
      \mathop{ > }\limits^{(i)} & \alpha _\varepsilon \big( 1 - \exp\left[ -L_\varepsilon ^2 \right] \big) \\
      \mathop{ \geq }\limits^{(ii)} & \alpha _\varepsilon \big( 1 - \exp\left[ 4g(\alpha _\varepsilon )/(1-2\alpha _\varepsilon ) \right] \big) \\
      =& h(\alpha _\varepsilon ) \\
      \mathop{ > }\limits^{(iii)} & \varepsilon ,
    \end{align*}
    where $ (i) $ from \autoref{equation:robust_testing_constant_condition_2}; $ (ii) $ from \autoref{equation:robust_testing_constant_condition_0}; $ (iii) $ from \autoref{equation:robust_testing_constant_condition_3}.
    \item[(ii)] Verify that $ (1/2-\alpha _\varepsilon )\log (1/\varrho) \geq -2g(\alpha _\varepsilon )$:
    \begin{align*}
      (1/2-\alpha _\varepsilon )\log (1/\varrho) >& (1/2-\alpha _\varepsilon )L_\varepsilon ^2 \mathop{ > }\limits^{(i)} -2g(\alpha _\varepsilon ), 
    \end{align*}
    where $ (i) $ from \autoref{equation:robust_testing_constant_condition_0}.
  \end{itemize}
  As a result, condition (i) and (ii) place the argument in the large-$ \frac{r}{\sigma } $ branch of their proof; and the quantification of the constants gives the exponent and threshold of the tail bound. We remind the readers that this only applies to the near-breakdown regime $ \varepsilon \in [1/4,1/2) $. Their original small-$ \frac{r}{\sigma } $ branch allows $ \varepsilon $ to be close to $ 0 $ where their $ r\gtrsim C_1(\kappa )\varepsilon \sigma  $ does not automatically imply $ C'r/\sigma > L $.

  Applying their Lemma B.3 and inequality (B.7), then following their procedure to compare the clean and corrupted voting results gives
  \begin{align*}
    \mathop{ \sup }\limits_{\mu \in B(\nu _1,r)} \mathbbm{P}_{ \mu ,\mathcal{C} }\left( \psi = 1 \right)  \leq \exp\left[ -\frac{1}{2}(1/2-\alpha _\varepsilon ) (C')^2N\frac{r^2}{\sigma ^2} \right] .
  \end{align*}
  Noticing that the definition of $ \alpha _\varepsilon  $ gives $ a_\varepsilon = \frac{1-2\alpha _\varepsilon }{2Q_3} $, setting $ Q_0 \leq Q_3(C')^2/2 $ gives the desired result. The other branch $ \mathbbm{P}_{ \mu \in B(\nu _2,r) }\left( \psi = 0 \right) $ follows similarly.

\end{proof}

\section{Minimax Upper Bound for Bounded Case}\label{appendix:minimax_upper_bound_bounded_case}
We now apply the robust testing procedure \autoref{lemma:robust_pairwise_testing_near_breakdown_point} to the tournament-on-a-tree procedure introduced in \autoref{appendix:tournament_on_tree_framework}. In the following, our focus is on the near-breakdown regime $ \varepsilon \in [1/4,1/2) $; for $ \varepsilon \in [0,1/4) $, the minimax upper bound follows from \cite{prasadan_information_2026}[Theorem 3.10].

Applying the robust pairwise testing to the tournament-on-the-tree procedure, we obtain a path along the tree $ \Y = [\Y_1,\Y_2,\ldots ] $ where $ \Y_{J+1} \in \mathcal{O}(\Y_J) $ is the winner of the tournament at the $ J $-th iteration, as introduced in \autoref{appendix:tournament_on_tree_framework}. Now we prove that by choosing a proper stopping step $ J $ as the output, the procedure achieves the minimax rate. For the bounded case, we use the following two stopping conditions:
  \begin{itemize}[topsep=2pt,itemsep=0pt]
    \item[\textbf{A}] $ Q_0a_\varepsilon Nr_J^2/\sigma ^2 \geq 4\log \mloc(\frac{c}{2}r_J,c) \vee \log 2 $;
    \item[\textbf{B}]$ r_J \geq D_1\varepsilon \sigma \sqrt{z_\varepsilon } $.
  \end{itemize}
  Here $ Q_0,D_1 $ are the constants in \autoref{lemma:robust_pairwise_testing_near_breakdown_point} and $ r_J = d/(2^{J-1}(C+1)) $ is the geometric scale as in \autoref{equation:geometric_scale_r_J}. Let $ J^A $ and $ J^B $ be the largest positive integers such that the conditions \textbf{A} and \textbf{B} are satisfied, respectively; if the corresponding condition is never satisfied, we set the index to be $ 1 $. This separation allows us to analyze the two possible first-failure scenarios, representing different statistical meanings: (\textbf{A}) the critical local radius controlled by the geometry of $ K $; (\textbf{B}) the contamination error. Then we let the final stopping step be $ J^* := J^A \wedge J^B $ and the final estimator be $ \hat{\mu }_{\mathrm{bounded}} := \Y_{J^*} $.

  The following lemma shows the error bound induced by the stopping conditions using \autoref{lemma:robust_pairwise_testing_near_breakdown_point}.

\begin{lemma}\label{lemma:robust_tournament_on_tree_near_breakdown_point_bounded_case}
  Let $ K $ be a bounded star-shaped set with diameter $ d $ and $ \mu \in K $, $ \varepsilon \in [1/4,1/2) $. Let $ r_J $ be the geometric scale as in \autoref{equation:geometric_scale_r_J}. If both conditions \textbf{A} and \textbf{B} are satisfied at some iteration $ J $, then 
  \begin{align*}
    \mathbbm{E}_{  }\left[ \left\Vert \Y_J - \mu  \right\Vert ^2 \right] \lesssim r_J^2 + \frac{\sigma ^2}{a_\varepsilon N}\exp\left[ -Q_0 a_\varepsilon N \frac{r_J^2}{2\sigma ^2} \right] \lesssim r_J^2.
  \end{align*}
\end{lemma}

\begin{proof}
  By the monotonicity of the local metric entropy as in \autoref{lemma:useful_properties}[(i)], for any $ 1\leq j\leq J $ we have
  \begin{align*}
    r_j \geq D_1\varepsilon \sigma \sqrt{z_\varepsilon },\quad Q_0\frac{a_\varepsilon Nr_j^2}{\sigma ^2} \geq 4\log \mloc(\frac{c}{2}r_j,c) \vee \log 2.
  \end{align*}
  As a result, \autoref{lemma:robust_pairwise_testing_near_breakdown_point} applies to every comparison before the stopping iteration $ j \leq J $.
  
  Now following the tree-tail calculation of \cite{prasadan_information_2026}[Lemma B.5, Lemma B.6] but with a different tail probability from \autoref{lemma:robust_pairwise_testing_near_breakdown_point}, we account for the two local packing factors in the union bound. We remark that the only differences are the error tail exponent $ Q_0a_\varepsilon Nr_j^2/\sigma ^2 $ and the threshold $ D_1\varepsilon \sigma \sqrt{z_\varepsilon } $. The union bound on the tree gives the tail probability
  \begin{align*}
    \mloc (\frac{c}{2}r_j,c)^2\exp\left[ -Q_0a_\varepsilon Nr_j^2/\sigma ^2 \right] \leq \exp\left[ -Q_0a_\varepsilon Nr_j^2/(2\sigma ^2) \right].
  \end{align*}
  The subsequent tail integration, as in their Lemma B.6, gives the error bound
  \begin{align*}
    \mathbbm{E}_{  }\left[ \left\Vert \Y_J - \mu  \right\Vert ^2 \right] \lesssim r_J^2 + \frac{\sigma ^2}{a_\varepsilon N}\exp\left[ -Q_0 a_\varepsilon N \frac{r_J^2}{2\sigma ^2} \right] . 
  \end{align*}
  Finally, the condition $ \frac{Q_0a_\varepsilon Nr_J^2}{\sigma ^2} \geq \log 2 $ of (\textbf{A}) implies that the second term is dominated by the first term, giving the final result.
  
\end{proof}

\begin{lemma}
  For $ \varepsilon \in [1/4,1/2) $, the final estimator $ \hat{\mu }_{\mathrm{bounded}} $ satisfies
  \begin{align*}
    \mathop{ \sup }\limits_{\mu \in K} \mathop{ \sup }\limits_{\mathcal{C}} \mathbbm{E}_{ \mu ,\mathcal{C} } \left[ \left\Vert \hat{\mu }_{\mathrm{bounded}} - \mu  \right\Vert ^2 \right] \lesssim ( r_{J^A}^2 \vee \sigma ^2\varepsilon ^2\log(1/(1-2\varepsilon )) ) \wedge d^2.
  \end{align*}
\end{lemma}

\begin{proof}
  We first deal with the case when at least one of the two conditions fails at $ J=1 $, where the estimator is set to be $ \hat{\mu }_{\text{bounded}} = \Y_1 $. In this case, since $ \mu ,\Y_1\in K $, naturally we have $ \left\Vert \Y_1 - \mu  \right\Vert ^2 \leq d^2 $. Also, we remind the readers that $ r_1 = d/(C+1) \asymp d $. Now we discuss the consequences of each failure:
  \begin{itemize}[topsep=2pt,itemsep=0pt]
    \item \textbf{If (A) fails at $ J=1 $}, then $ J^A = J^* = 1 $ and we have $ d^2 \asymp r_{J^A}^2 $. Thus
    \begin{align*}
       \mathbbm{E}_{  }\left[ \left\Vert \hat{\mu }_{\mathrm{bounded}} - \mu  \right\Vert ^2 \right] \leq d^2 \asymp r_{J^A}^2 \wedge d^2 \leq ( r_{J^A}^2 \vee \sigma ^2\varepsilon ^2\log(1/(1-2\varepsilon )) ) \wedge d^2.
    \end{align*}
    \item \textbf{If (B) fails at $ J=1 $}, then we have $ d^2\asymp r_1^2 \lesssim \varepsilon ^2\sigma ^2\log(1/(1-2\varepsilon )) $. Thus
    \begin{align*}
      \mathbbm{E}_{  }\left[ \left\Vert \hat{\mu }_{\mathrm{bounded}} - \mu  \right\Vert ^2 \right] \leq d^2 \lesssim \varepsilon ^2\sigma ^2\log(1/(1-2\varepsilon )) \wedge d^2 \leq ( r_{J^A}^2 \vee \sigma ^2\varepsilon ^2\log(1/(1-2\varepsilon )) ) \wedge d^2.
    \end{align*}
  \end{itemize}

  It remains to consider the case when both conditions are satisfied at the root $ J=1 $. By the preceding lemma, we have
  \begin{align*}
     \mathop{ \sup }\limits_{\mu \in K} \mathop{ \sup }\limits_{\mathcal{C}} \mathbbm{E}_{ \mu ,\mathcal{C} } \left[ \left\Vert \hat{\mu }_{\mathrm{bounded}} - \mu  \right\Vert ^2 \right] \lesssim r_{J^*}^2.
  \end{align*}
  At the level when condition (\textbf{B}) fails, we have $ r_{J^B} < 2D_1\varepsilon \sigma \sqrt{z_\varepsilon } $. Then since $ J\mapsto r_J $ is decreasing, 
  \begin{align*}
    r_{J^*}^2 = r_{J^A}^2 \vee r_{J^B}^2 \lesssim  r_{J^A}^2 \vee \sigma ^2\varepsilon ^2\log(1/(1-2\varepsilon )).
  \end{align*}
  Combining the two cases with the trivial bound $ d^2 $ we obtain the desired result.
  
\end{proof}

\begin{proof}[Proof of \autoref{theorem:minimax_upper_bound_bounded}]
  Recall that
  \begin{align*}
    \delta _\varepsilon^* = \sup\big\{ r\geq 0: a_\varepsilon N\frac{r^2}{\sigma ^2} \leq \log \mloc(r,c) \}.
  \end{align*}
  If $ K $ is a singleton, then $ \delta _\varepsilon ^* = 0 $ and the result is trivial. Otherwise, we split the discussion by the value of $ \kappa _\varepsilon := a_\varepsilon N\delta _\varepsilon ^{*2} / \sigma ^2 $ and argue that the rate given in the previous lemma matches the desired minimax rate. 

  \begin{itemize}[topsep=2pt,itemsep=0pt]
    \item \textbf{Case 1: }If $ \kappa _\varepsilon \leq 8\log 2 $. By the monotonicity of the local metric entropy \autoref{lemma:useful_properties}[(i)], we have 
    \begin{align*}
      8\log 2 \geq \kappa _\varepsilon = \frac{1}{4}a_\varepsilon N\frac{(2\delta _\varepsilon ^*)^2}{\sigma ^2} \geq \frac{1}{4} \log \mloc(2\delta _\varepsilon ^*,c) .
    \end{align*}
    By the same argument as in \cite{prasadan_information_2026}[proof of Theorem 3.10, Case 3], which utilizes the fact that there must be line segments in $ K $ of length at least $ d/3 $ \autoref{lemma:useful_properties}[(iii)], we have $ d\leq 12\delta _\varepsilon ^* $. As a result, we directly have
    \begin{align*}
       \mathbbm{E}_{  }\left[ \left\Vert \hat{\mu }_{\mathrm{bounded}} - \mu  \right\Vert ^2 \right] \leq d^2 \asymp (\delta _\varepsilon ^{*2} \vee \sigma ^2\varepsilon ^2\log(1/(1-2\varepsilon ))) \wedge d^2.
    \end{align*}
    \item \textbf{Case 2:} If $ \kappa _\varepsilon > 8\log 2 $. To apply the previous lemma, it suffices to show that $ \delta _\varepsilon ^* \gtrsim r_{J^A} $. We do so by using an intermediate quantity $ \tilde{r} := D_4\delta_\varepsilon ^* $ and arguing that $ \tilde{r} \gtrsim r_{J^A} $. Here we set
    \begin{align*}
      D_3 := \frac{c}{4} \wedge \frac{\sqrt{Q_0}}{4}, \quad D_4:= 2/D_3. 
    \end{align*}
    Consequently, we have $ D_3 \leq c/2 $, $ 4D_3^2 < Q_0 $, $ D_3D_4 = 2 $, and $ Q_0D_4^2 \geq 64$. Now we verify that $ \tilde{r} $ satisfies the condition (\textbf{A}):
    We have
    \begin{align*}
      \frac{Q_0a_\varepsilon N (D_4\delta _\varepsilon ^*)^2}{\sigma ^2} = & \frac{Q_0}{D_3^2} \frac{a_\varepsilon N (2\delta _\varepsilon ^*)^2}{\sigma ^2} \\
      > & \frac{Q_0}{D_3^2} \log \mloc(2\delta _\varepsilon ^*,c) \\
      \mathop{ \geq }\limits^{(i)}  & \frac{Q_0}{D_3^2} \log \mloc( \frac{cD_4}{2}\delta _\varepsilon ^*,c) \\
      =& \frac{Q_0}{D_3^2} \log \mloc( \frac{c}{2}\tilde{r},c) \\
      \geq& 4\log \mloc( \frac{c}{2}\tilde{r},c),
    \end{align*}
    where $ (i) $ from the monotonicity of the local metric entropy \autoref{lemma:useful_properties}[(i)] and the fact that $ D_3\leq c/2 $. Also from the first line we have
    \begin{align*}
       \frac{Q_0a_\varepsilon N (D_4\delta _\varepsilon ^*)^2}{\sigma ^2} = & Q_0D_4^2\kappa _\varepsilon > 64\log 2 > \log 2.
    \end{align*}

    Now if $ r_1 < \tilde{r} $, then we have $ r_{J^A} \leq r_1 < \tilde{r} $. Otherwise, we can choose some $ J $ such that $ \tilde{r} \leq r_J < 2\tilde{r} $. By the monotonicity of $ r \mapsto Q_0a_\varepsilon Nr^2/\sigma ^2 - \big(4\log \mloc(\frac{c}{2}r,c) \vee \log 2\big)$, we must have $ r_{J^A} \leq r_J < 2\tilde{r} $. As a result, both cases give $ r_{J^A} \lesssim \tilde{r} $. Substituting this into the previous lemma gives
    \begin{align*}
      \sup _{\mu \in K} \sup_{\mathcal{C}} \mathbbm{E}_{  }\left[ \left\Vert \hat{\mu }_{\mathrm{bounded}} - \mu  \right\Vert ^2 \right]  \lesssim &(r_{J^A}^2 \vee \sigma ^2\varepsilon ^2\log(1/(1-2\varepsilon ))) \wedge d^2 \\
      \lesssim &(\delta _\varepsilon ^{*2} \vee \sigma ^2\varepsilon ^2\log(1/(1-2\varepsilon ))) \wedge d^2 .
    \end{align*}    
  \end{itemize}
  When $ \varepsilon \geq 1/4 $, replacing $ \log(1/(1-2\varepsilon )) $ by $ \log(e/(1-2\varepsilon )) $ does not change the order, giving the desired result in this regime. The case $ \varepsilon \in [0,1/4) $ follows from \autoref{proposition:agreement_with_existing_result}.
  
\end{proof}

\section{Minimax Upper Bound for Unbounded Case}\label{appendix:minimax_upper_bound_unbounded_case}

For the unbounded case, the tree is preceded by a localization step, as described in \autoref{appendix:tournament_on_tree_framework}. The localization and tree construction remain unchanged up to constant rescaling, and it suffices to analyze the constant calibration in these steps. The objective is to restore the dependence on $ \varepsilon  $ of the original proof of \cite{prasadan_information_2026} where $ \varepsilon  $ is assumed to be bounded away from $ 1/2 $. In the remaining part we only consider the near-breakdown regime $ \varepsilon \in [1/4,1/2) $; for $ \varepsilon \in [0,1/4) $, the minimax upper bound follows using similar arguments as in bounded case.

\subsection{Localization Step}

The two requirements on $ \gamma _\varepsilon  $ inherited from \cite{prasadan_information_2026}[Section 5.2] are
\begin{align*}
  \gamma _\varepsilon  > \max \big( 1-\dfrac{ C_{3,\varepsilon } }{ 6(C+1)^2\log2 } , 1-\dfrac{ 4 }{(C+1)^2C_{1,\varepsilon }^2 } \big),
\end{align*}
where we note that $ C_{1,\varepsilon }^2\asymp z_\varepsilon = \log(e/(1-2\varepsilon )) $ and $ C_{3,\varepsilon }\asymp a_\varepsilon = 1-2\varepsilon $. We meet these requirements by setting
\begin{align}\label{equation:localization_step_constant_calibration}
  \gamma _\varepsilon := 1 - Q_4a_\varepsilon 
\end{align}
where $ Q_4>0 $ can be chosen universally since $ a_\varepsilon z_\varepsilon \lesssim 1 $. As a result we have $ \gamma_\varepsilon  \asymp 1 $, $ 1-\gamma_\varepsilon \asymp a_\varepsilon  $. Without loss of generality, we choose $ Q_4 \leq 1 $ so that $ \gamma _\varepsilon \in [1/2,1) $ for all $ \varepsilon \in [1/4,1/2) $. 

Now we proceed to the localization radius $ R_\varepsilon  $. Inherited from \cite{prasadan_information_2026}[Equation (5.1)], we require that
\begin{align}\label{equation:condition_on_R_epsilon}
  R_\varepsilon ^2 > \max\big(   \frac{8n\sigma^2\log5}{1-\gamma_\varepsilon}, \frac{-64\sigma^2g(\varepsilon)}{\gamma_\varepsilon a_\varepsilon }, \frac{1024\sigma^2\log2}{\gamma_\varepsilon a_\varepsilon}, 4(C+1)^2 \big) .
\end{align}
By \autoref{equation:function_g_h_endpoint_behavior_1}, we have $ -g(\varepsilon)/a_\varepsilon \asymp z_\varepsilon  $. Since $ n\geq 1 $ and $ a_\varepsilon z_\varepsilon \lesssim 1 $, the above condition can be satisfied by setting
\begin{align}\label{equation:localization_step_radius_calibration}
    R_\varepsilon ^2 := Q_5\dfrac{ n\sigma ^2 }{ a_\varepsilon } \vee 8(C+1)^2, 
\end{align}
where $ Q_5>0 $ is a sufficiently large universal constant.

For every $t\geq R_\varepsilon$, the conditions \autoref{equation:condition_on_R_epsilon} for
$R_\varepsilon$ remain valid with $t$ in place of $R_\varepsilon$. Applying the same argument as in \cite{prasadan_information_2026}[Lemma 5.3] gives
\begin{align}\label{equation:localization_step_success_probability}
  \mathbbm{P}_{  }\left( \mu \not\in S(t ) \right) \leq \exp\left(-\frac{Na_\varepsilon\gamma_\varepsilon}{64}
  \frac{t^2}{\sigma^2}
  \right),\quad t\geq R_\varepsilon.
\end{align}
It follows from \autoref{equation:localization_step_radius_calibration} that the right-hand side is bounded by $ \exp[-Q_5Nn/128] $. As a result, the localization step succeeds with high probability.

In \autoref{appendix:tournament_on_tree_framework}, we have described the construction of the estimator for the unbounded case depending on whether $ S(R_\varepsilon ) $ is empty or not. 

\subsection{Risk Upper Bound Analysis}

Recall that the final estimator is defined as
\begin{itemize}[topsep=2pt,itemsep=0pt]
  \item If $ S(R_\varepsilon ) \neq \varnothing $: the algorithm outputs a winner chain $ \Y^s = [\Y_1^s,\Y_2^s,\ldots ] $ from the tournament-on-a-tree procedure on the localized set $ S(R_\varepsilon ) $; traversing the tree until some stopping step $ J^* $ gives the final estimator $ \Y_{J^*}^s $;
  \item If $ S(R_\varepsilon ) = \varnothing $: let $ \hat{R} := \inf\{R\geq0: S(R)\neq \varnothing \} $ and the estimator is $ \min_{\text{lex}}S(\hat{R}) $.
\end{itemize}
We denote the final estimator as $ \hat{\mu }_{\mathrm{unbounded}} $. 

We now state the stopping conditions for the first case. Here we slightly overload the notation $ r_J $ with $ d $ replaced by $ d_m $:
\begin{align*}
  r_J:=\dfrac{ d_m }{ 2^{J-1}(C+1) },\quad d_m := 2m+2R_\varepsilon = \frac{2c}{c-2}R_\varepsilon .  
\end{align*}
The stopping conditions are
\begin{itemize}[topsep=2pt,itemsep=0pt]
  \item[\textbf{A}] $ Q_0a_\varepsilon Nr_J^2/\sigma ^2 \geq 6\log \mloc(\frac{c}{2}r_J,c) \vee \log 2 $;
  \item[\textbf{B}]$ r_J \geq D_1\varepsilon \sigma \sqrt{z_\varepsilon } $.
\end{itemize}
Let $ J^A $ and $ J^B $ be the largest positive integers such that the conditions \textbf{A} and \textbf{B} are satisfied, respectively; if the corresponding condition is never satisfied, we set the index to be $ 1 $. Then we let the stopping step be $ J^* := J^A \wedge J^B $. As a result, the estimator is
\begin{align*}
  \hat{\mu }_{\mathrm{unbounded}} = \begin{cases}
    \Y_{J^*}^s, & S(R_\varepsilon ) \neq \varnothing ,\\
    \min_{\text{lex}}S(\hat{R}), & S(R_\varepsilon ) = \varnothing .
  \end{cases}
\end{align*}

\begin{lemma}
  If $ \varepsilon \in[1/4,1/2) $, then for any $ J $ such that both conditions \textbf{A} and \textbf{B} are satisfied, we have
  \begin{align*}
     \mathbbm{E}_{  }\left[ \left\Vert \hat{\mu }_{\mathrm{unbounded}} -\mu  \right\Vert ^2  \right] \lesssim r_J^2.
  \end{align*}
\end{lemma}

\begin{proof}
  The proof is divided into two parts according to whether $ S(R_\varepsilon ) $ is empty, following the idea of \cite{prasadan_information_2026}[Lemma D.1, Lemma D.2]. In their proof of Lemma D.2, Parts I-IV correspond to the case $ S(R_\varepsilon ) \neq \varnothing $ and Part V corresponds to the case $ S(R_\varepsilon ) = \varnothing $; Part VI combines the two contributions to give the final risk bound. We refer the readers to their proof for the details of the tree traversal and tail probability calculation since the geometric argument is unchanged, and we have calibrated the involved $ \varepsilon  $-dependent parameters. The main steps are summarized below.

  We begin with a counterpart of their Lemma D.1 that gives the tail probability of the tree traversal with constants uniform
  over $\varepsilon\in[1/4,1/2)$. The difference from the bounded case is that we now need to account for the data-dependent choice of the root $ s $. This requires an extra probability bound to address the localization step, which is given in \autoref{equation:localization_step_success_probability}. Repeating the same steps gives
  \begin{align*}
    \mathbbm{P}_{  }\left( \left\Vert \Y_J^s - \mu  \right\Vert > (C+1)r_J | S(R_\varepsilon ) \neq \varnothing \right) \leq& 2\exp\left[ -\frac{Na_\varepsilon\gamma_\varepsilon R_\varepsilon^2}{64\sigma^2} \right] + 4\exp\left[ -\frac{Q_0a_\varepsilon Nr_J^2}{2\sigma^2} \right]\\
    \lesssim & \exp\left[ -\frac{Q_6a_\varepsilon Nr_J^2}{\sigma^2} \right],
  \end{align*}
  where $ Q_6 $ is a universal constant, for example we can set $ Q_6 = \frac{Q_0}{2} \wedge \frac{(2C+1)^2}{2048} $. Since $ r_J\leq r_1 = 4R_\varepsilon /(2C+1) $ and $ \gamma _\varepsilon \geq 1/2 $ the inequality holds for all $ \varepsilon \in [1/4,1/2) $. 
  
  The $ S(R_\varepsilon ) \neq \varnothing $ case, where the estimator is $ \Y^s_{J^*} $, has the same structure as \autoref{lemma:robust_tournament_on_tree_near_breakdown_point_bounded_case} by assessing the tail probability of the tree traversal and then integrating the tail. To obtain the conditional expected risk, we apply the preceding tail bound at every level $j\leq J$ and follow Parts I-IV of \cite{prasadan_information_2026}[Lemma D.2]. Specifically, there are two scales involved in the integration (say, the integration index is $ t $): the tree argument controls the intermediate $ r_J\lesssim t \lesssim R_\varepsilon  $ regime; the nesting property and the localization bound \autoref{equation:localization_step_success_probability} control the large $ t\gtrsim R_\varepsilon  $ regime. The calculation yields the following: for any $ J $ such that both conditions \textbf{A} and \textbf{B} are satisfied, we have
  \begin{align}\label{equation:unbounded_case_upper_bound_case_1}
    \mathbbm{E}_{  }\left[ \left\Vert \Y^s_{J} - \mu  \right\Vert ^2 | S(R_\varepsilon ) \neq \varnothing \right] \lesssim r_J^2 + \frac{\sigma ^2}{a_\varepsilon N}\exp\left[ -Q_6 a_\varepsilon N \frac{r_J^2}{\sigma ^2} \right] \lesssim r_J^2,
  \end{align}
  where the last inequality follows from condition \textbf{A}.
  Since both stopping conditions hold at $J$, we have $J\le J^*$. By \autoref{lemma:pruned_directed_tree_properties}[(iv)] that $ \left\Vert \Y_{J^*}^s - \Y_J^s \right\Vert \leq\frac{4C+5}{2}r_J $ as long as $ J\leq J^* $, we obtain
  \begin{align*}
     \mathbb E\left[ \left\Vert \Y_{J^*}^s - \mu  \right\Vert ^2 | S(R_\varepsilon ) \neq \varnothing \right] \lesssim r_J^2.
  \end{align*}
  
  The $ S(R_\varepsilon ) = \varnothing $ case involves a peeling argument to break the event $ S(R_\varepsilon ) = \varnothing $ into a disjoint union:
  \begin{align*}
    \{S(R_\varepsilon )= \varnothing\} = \{ \hat{R} > R_\varepsilon \}= \uplus_{l=1}^{\infty }\{ 2^{l-1}R_\varepsilon < \hat{R} \leq 2^lR_\varepsilon \}.
  \end{align*}
  Then we let $ p_l := \mathbbm{P}_{  }\left( 2^{l-1}R_\varepsilon < \hat{R} \leq 2^lR_\varepsilon \right) $ and apply the localization tail probability \autoref{equation:localization_step_success_probability} to obtain
  \begin{align}\label{equation:p_l_tail_bound}
    p_l \leq& \mathbbm{P}\bigl(\hat R>2^{l-1}R_\varepsilon\bigr)\notag \\
    \leq& \mathbbm{P}\bigl(\mu\notin S(2^{l-1}R_\varepsilon)\bigr) \notag \\
    \leq& \exp\left[ -\frac{Na_\varepsilon \gamma _\varepsilon }{64}\frac{4^{l-1}R_\varepsilon ^2}{\sigma ^2} \right]:= \exp(- t_l),
  \end{align}
  where we remind the readers that $ t_l \geq t_1 > 1 $ by \autoref{equation:localization_step_radius_calibration}. Then following the same integration argument gives a counterpart of \cite{prasadan_information_2026}[Equation (D.8)]:
  \begin{align*}
    \mathbbm{E}_{  }\left[ \left\Vert \hat{\mu }_{\mathrm{unbounded}} - \mu  \right\Vert ^2 | 2^{l-1}R_\varepsilon < \hat{R} \leq 2^lR_\varepsilon \right] \lesssim 4^{l-1}R_\varepsilon ^2 \vee \frac{\sigma ^2}{Na_\varepsilon \gamma _\varepsilon }\log\frac{1}{p_l} + \frac{\sigma ^2}{Na_\varepsilon \gamma _\varepsilon }.
  \end{align*}
  Substituting the tail bound of $ p_l $ gives
  \begin{align}\label{equation:unbounded_case_upper_bound_case_2}
    \mathbbm{E}_{  }\left[ \left\Vert \hat{\mu }_{\mathrm{unbounded}} - \mu  \right\Vert ^2 \mathbbm{1}\big(S(R_\varepsilon ) = \varnothing\big) \right] =& \sum_{l=1}^{\infty } p_l\mathbbm{E}_{  }\left[ \left\Vert \hat{\mu }_{\mathrm{unbounded}} - \mu  \right\Vert ^2 | 2^{l-1}R_\varepsilon < \hat{R} \leq 2^lR_\varepsilon \right] \notag\\
    \lesssim &\sum_{l=1,p_l>0}^{\infty } p_l \left( 4^{l-1}R_\varepsilon ^2 \vee \frac{\sigma ^2}{Na_\varepsilon \gamma _\varepsilon }\log\frac{1}{p_l} \right) + \frac{\sigma ^2}{Na_\varepsilon \gamma _\varepsilon } \notag\\   
    \mathop{ \lesssim }\limits^{(i)} &  \frac{\sigma ^2}{Na_\varepsilon \gamma _\varepsilon }\sum_{l=1}^{\infty } t_l\exp(-t_l) + \frac{\sigma ^2}{Na_\varepsilon \gamma _\varepsilon } \notag\\
    \mathop{ \lesssim }\limits^{(ii)}  &\frac{\sigma ^2}{Na_\varepsilon \gamma _\varepsilon }\asymp \frac{\sigma ^2}{Na_\varepsilon },
  \end{align}
  where $ (i) $ by rearranging \autoref{equation:p_l_tail_bound} to see that $ p_l4^{l-1}R_\varepsilon ^2 \leq 64\sigma ^2/(Na_\varepsilon \gamma _\varepsilon )t_l\exp(-t_l) $; also  we use the monotonicity of $ x\mapsto x\log(1/x) $ for $ x\in(0,1/e) $ and thus $ p_l\log(1/p_l) \leq t_l\exp(-t_l) $ for all $ l\geq 1 $. $ (ii) $ since $ t_l $ is a geometrically increasing sequence.
  
  Finally, we combine the two cases \autoref{equation:unbounded_case_upper_bound_case_1} and \autoref{equation:unbounded_case_upper_bound_case_2}. Note that condition (\textbf{A}) implies that $ r_J^2\gtrsim \frac{\sigma ^2}{Na_\varepsilon } $. Thus we have
  \begin{align*}
     \mathbbm{E}_{  }\left[ \left\Vert \hat{\mu }_{\mathrm{unbounded}} - \mu  \right\Vert ^2 \right] \lesssim r_{J}^2 + \frac{\sigma ^2}{Na_\varepsilon } \lesssim r_J^2.
  \end{align*}
\end{proof}

\begin{lemma}
  If $ \varepsilon \in[1/4,1/2) $, then $ \hat{\mu }_{\mathrm{unbounded}} $ satisfies
  \begin{align*}
    \mathop{ \sup }\limits_{\mu \in K} \mathop{ \sup }\limits_{\mathcal{C}} \mathbbm{E}_{ \mu ,\mathcal{C} } \left[ \left\Vert \hat{\mu }_{\mathrm{unbounded}} - \mu  \right\Vert ^2 \right] \lesssim  r_{J^A}^2 \vee \sigma ^2\varepsilon ^2\log(1/(1-2\varepsilon )) .
  \end{align*}
\end{lemma}

\begin{proof}
  We first check that both conditions \textbf{A} and \textbf{B} are satisfied at the root $ J=1 $. Indeed, we have $ r_1 = \frac{d_m}{C+1} = \frac{4}{2C+1}R_\varepsilon  $. By \autoref{equation:localization_step_radius_calibration}, we have $ r_1^2\gtrsim n\sigma ^2/a_\varepsilon   $. Since $ n\gtrsim \log\mloc(\frac{c}{2}r_1,c) $ by \autoref{lemma:useful_properties}[(ii)] and $ a_\varepsilon z_\varepsilon \lesssim 1 $, by taking $ Q_5 $ sufficiently large, we must have both conditions satisfied at the root, and hence at $ J^* $. As a result, the previous lemma applies, and we have
  \begin{align*}
    \mathop{ \sup }\limits_{\mu \in K} \mathop{ \sup }\limits_{\mathcal{C}} \mathbbm{E}_{ \mu ,\mathcal{C} } \left[ \left\Vert \hat{\mu }_{\mathrm{unbounded}} - \mu  \right\Vert ^2 \right] \lesssim r_{J^*}^2.  
  \end{align*}
  At step $ J^B+1 $ condition \textbf{B} fails, which gives $ r_{J^B}/2 = r_{J^B+1} < D_1\varepsilon \sigma \sqrt{z_\varepsilon } $. Since $ J^* = J^A \wedge J^B $, we have $ r_{J^*}^2 = r_{J^A}^2 \vee r_{J^B}^2 \lesssim r_{J^A}^2 \vee \sigma ^2\varepsilon ^2\log(1/(1-2\varepsilon )) $. Substituting this into the previous display gives the desired result.
  
\end{proof}

\begin{proof}[Proof of \autoref{theorem:minimax_upper_bound_unbounded}]
  The regime $ \varepsilon \in [0,1/4) $ follows from \cite{prasadan_information_2026}[Theorem 5.9] and the rate matching argument as in \autoref{proposition:agreement_with_existing_result}. For $ \varepsilon \in [1/4,1/2) $, the logic is similar to that in the proof of \autoref{theorem:minimax_upper_bound_bounded}. Here since $ K $ is unbounded, the argument in Case 1 there does not apply as $ K $ contains arbitrarily long line segments, meaning that we always have $ \kappa _\varepsilon := a_\varepsilon N\delta _\varepsilon ^{*2}/\sigma ^2 $ large enough and hence the discussion falls into Case 2. Then it suffices to complete an analogue of Case 2 of the proof, arguing that $ r_{J^A} \lesssim \delta _\varepsilon ^* $. The details are omitted since the argument is identical. Gathering the two cases gives the desired result for all $ \varepsilon \in [0,1/2) $.
\end{proof}

\end{document}